\documentclass[12pt]{article}
\usepackage{amsmath, amssymb, amsthm, mathrsfs, constants, keyval}

\newconstantfamily{low}{symbol=c}

\newcommand{\clabel}[1]{\Cl[low]{#1}}
\newcommand{\cref}[1]{\Cr{#1}}

\theoremstyle{plain}        

\newtheorem{theorem}{Theorem}
\newtheorem{corollary}{Corollary}
\newtheorem{proposition}{Proposition}
\newtheorem{lemma}{Lemma}
\theoremstyle{remark}        
\newtheorem{remark}{Remark}
\theoremstyle{definition}

\begin{document}

\title{Existence, Uniqueness, Regularity and Long-term Behavior for Dissipative Systems Modeling Electrohydrodynamics}
\author{Rolf J. Ryham}
\maketitle

\abstract{We study a 
dissipative
system of nonlinear and nonlocal equations
modeling the flow of
electrohydrodynamics.  
The existence, uniqueness and regularity of solutions is proven
for general $\mathbf{L}^2$ initial data in two space dimensions
and for small data in data in three space dimensions.  
The existence in three dimensions is established by studying a linearization
of a relative entropy functional.  We also establish the convergence to 
the stationary solution with a rate.}

\section{Introduction}
In this paper, we study the following nonlinear system of equations;
\begin{gather}
\label{NS}
{\bf u}_t +  {\bf u} \cdot  \nabla {\bf u} + \nabla p =  \Delta {\bf u} + \Delta \phi \nabla \phi,\\
\label{DIV}
\nabla \cdot {\bf u}  = 0,\\
\label{NP1}
v_t + {\bf u} \cdot \nabla v  = \nabla \cdot \left(\nabla v -v \nabla \phi \right),\\
\label{NP2}
w_t + {\bf u} \cdot \nabla w = \nabla \cdot \left( \nabla w +w \nabla \phi \right),\\
\label{poisson}
 \Delta \phi = v - w
\end{gather}
in $\Omega \times (0,\infty)$ for a connected, bounded, open subset $\Omega$ of $\mathbf{R}^n$
with smooth boundary $\partial \Omega.$
Here ${\bf u}(x,t)$ is a vector in $\mathbf{R}^n$ 
and $p(x,t),$ $v(x,t),$ $w(x,t)$ and $\phi(x,t)$ are scalars.
Equation \eqref{NS} is the force balance equation of a  
viscous, incompressible fluid with velocity ${\bf u}$ and incompressibility condition \eqref{DIV}.
These are coupled with conservation 
equations \eqref{NP1}-\eqref{NP2} of a binary system of charges with densities $v, w$ and
the electric potential $\phi$ determined by the Poisson equation \eqref{poisson}. 
The force exerted by
the charged particles  on the fluid is 
$\Delta \phi \nabla \phi = \nabla \cdot  \sigma
  (=\sum_{i=1}^n (\sigma_{ij})_{x_i})$ 
where the electric stress $\mathbf{\sigma}$ is a rank one tensor plus a pressure;
for $i,j = 1,\dots,n,$ 
\begin{equation*}
[\sigma]_{ij} = \left(\nabla \phi \otimes \nabla \phi - \frac{1}{2}|\nabla \phi|^2I\right)_{ij}
 = \phi_{x_i} \phi_{x_j} - \frac{1}{2}|\nabla \phi|^2 \delta_{ij}.
\end{equation*}
The electric stress $\sigma$ stems from the balance of kinetic energy
with electrostatic energy via the least action principle, \cite{RyLiZi07}.   For simplicity, we have
assumed that the fluid density, viscosity, charge mobility and dielectric constant are unity.

Solutions for the velocity field equation 
are determined by the Dirichlet condition
\begin{equation}
\label{S}
{\bf u}(x,t) = 0   \quad \mbox{ for } (x,t) \in\partial {\Omega} \times (0,\infty).
\end{equation}
Solutions of the equations for the charges 
are determined by the natural (no flux) boundary conditions
\begin{gather}
\label{F1}
\frac{\partial v}{\partial \nu} -v \frac{\partial \phi}{\partial \nu} = 0,\quad 
\frac{\partial w}{\partial \nu} +w \frac{\partial \phi}{\partial \nu} = 0, \quad 
\mbox{ on } \partial \Omega \times (0,\infty),
\end{gather}
where $\nu$ is the outward pointing normal to $\partial \Omega.$
Along with \eqref{S}, equations \eqref{F1} are the natural
boundary condition of \eqref{NP1} and \eqref{NP2};
there is no flux of the charges through the boundary.
The integral of $v$ and $w$ are conserved quantities;
\begin{equation}
\label{conserved}
\frac{d}{dt} \int_{\Omega} v \, \mathrm{d}x
=
\frac{d}{dt} \int_{\Omega} w \, \mathrm{d}x
= 0.
\end{equation}
We assume that 
\begin{equation}
\label{posmass}
0 < \int_{\Omega} v_0 \, \mathrm{d}x, \int_{\Omega} w_0 \, \mathrm{d}x < \rho_0 < \infty.
\end{equation}
The constant $\rho_0$ is the characteristic charge.  The small
data results in section \ref{globalbehave} will in part
be formulated in terms of the size of $\rho_0.$

Solutions of   \eqref{poisson} are determined by 
\begin{equation}
\label{D}
\phi(x,t) = 0  \quad \mbox{ for } (x,t) \in\partial {\Omega} \times (0,\infty).
\end{equation}
Condition \eqref{D} states that the boundary of the domain is held at a fixed potential.
The evolution is determined by initial conditions 
\begin{equation}
\label{I}
\begin{split}
&\qquad \qquad \quad {\bf u}(x,0) = {\bf u}_0(x),  \quad \nabla \cdot {\bf u}_0(x) = 0, \\
&v(x,0)  = v_0(x) \geq 0, \quad w(x,0)  =w_0(x) \geq 0\quad \mbox{ for } x \in \Omega. 
\end{split}
\end{equation}

Since the Navier-Stokes equation is a subsystem of \eqref{NS}-\eqref{poisson},
one cannot expect better results than for the Navier-Stokes equations. 
In the absence of a fluid, the hydrodynamic system \eqref{NS}-\eqref{poisson} reduces to
the subsystem 
\begin{equation}
\label{DebyeHueckel}
\left\{
\begin{aligned}
&v_t  = \nabla \cdot \left(\nabla v -v \nabla \phi \right),\\
&w_t  = \nabla \cdot \left( \nabla w +w \nabla \phi \right),\\
&\Delta \phi = v - w.
\end{aligned}
\right.
\end{equation}
The equations \eqref{DebyeHueckel} are the Debye-H\"uckel 
system, a basic model  for the diffusion of ions in an electrolyte filling all of $\mathbf{R}^3$
first studied by W. Nernst and M. Plank at the end of the nineteenth century, \cite{DeHu23}.
The results for the Debye-H\"uckel system are complementary to those of Navier-Stokes.
In \cite{BiHeNa94}, Biler et al. proved the 
existence of global weak solutions in dimensions two 
and proved the uniqueness and 
regularity of local solutions in all dimensions under
appropriate assumptions. 
It is not known whether \eqref{DebyeHueckel} 
in general possesses global weak solutions in dimensions greater than two.  
In view of this challenge, one cannot expect solutions of 
\eqref{NS}-\eqref{poisson} to exist in dimensions greater than 2
for general data, c.f. theorems \ref{globalex} and \ref{globalex3d}. 

There are several theoreical difficulties associated with the system 
\eqref{NS}-\eqref{poisson}. 
The coefficient $\nabla \phi$ in the boundary value 
problem \eqref{NP1}-\eqref{NP2} and \eqref{F1} is determined
by nonlocal information.  The right-hand side of 
\eqref{NP1} and \eqref{NP2} cannot be formulated
in terms of a Frechet derivative of a functional
(c.f. the Erikson-Leslie theory for liquid crystals
and the Allen-Cahn and Cahn-Hilliard theory for fluid/interface
motion).  Furthermore, it is not clear how the basic energy law
for \eqref{NS}-\eqref{poisson} (c.f. inequality \eqref{apriori})  
implies the extension property for local solutions in dimensions 
greater than two.   Many of the standard techniques 
for parabolic PDE, e.g. the maximum principle and apriori estimates, 
are difficult to apply.  

Biler et al. presented the first mathematical existence,
uniqueness and regularity results in \cite{BiHeNa94}.  
They established the $\mathbf{L}^2$ 
convergence of solutions to the 
stationary solution without a rate.  In \cite{BiDo00},
this result was improved be establishing an exponential
$\mathbf{L}^1$ convergence with a rate depending only on $\Omega.$  Their
work relies heavily on the tools developed in 
\cite{ArMaToUn01} and \cite{UnArMaTo00} for the Fokker-Plank
equation.  In \cite{BeMeVa04}, an exponential 
$\mathbf{L}^2$ convergence result by means of a 
linearization of the an appropriate energy functional
was proved.  

The hydrodynamic setting presented here has been studied
in \cite{FaGa09}, 
\cite{Li09} 
and \cite{RyLiWa06}.  
The work of \cite{FaGa09} establishes several important
estimates for the hydrodynamical system when the 
boundary is assumed to be electrically insulated.
The work of \cite{Li09} studies the interesting
zero-dielectric limit of the system 
on the flat $n$-dimensional
torus.

\subsection{Basic Energy Law}
We develop the basic energy law for electrohydrodynamics.
Let us consider a classical solution 
${\bf u}, v, w$ of \eqref{NS}-\eqref{I}
on $\Omega \times (0,T).$  
Assume that $v$ and $w$ are positive
on $\overline{\Omega}.$ 
Throughout the paper, define 
\begin{equation*}
\psi_r(s) = s \log (s/r) - s + r, \quad r \in (0,\infty), s \in [0,\infty)
\end{equation*}
and define 
\begin{equation*}
\psi(\cdot) = \psi_1(\cdot)
\end{equation*}
The following energy functional will play an important role;
\begin{equation}
\label{ENTROPY}
{W} \equiv
\int_{\Omega} \psi(v) + \psi(w) \,dx   + \frac{1}{2}\|\nabla \phi\|_{\mathbf{2}}^2 + \frac{1}{2}\|{\bf u}\|_{\mathbf{L}^2}^2.
\end{equation}
The first two terms in this definition are the entropy of the charges $v$ and $w$ respectively,
while the last two are the electric energy of the charges and the kinetic energy of the fluid
respectively. 

Differenting $W$ 
with respect to $t$ gives
\begin{equation*}
\frac{dW}{dt} = (v_t , \psi'(v))   +  (w_t , \psi'(w))  + 
(\nabla \phi_t,  \nabla \phi)  +   ({\bf u}_t, { \bf u})
\end{equation*}
Here $(\cdot, \cdot)$ denotes the usual ${\bf L}^2$ inner product on $\Omega.$
Integrating by parts and using \eqref{D},  we see that
$(\nabla \phi_t, \nabla \phi ) = - (\Delta \phi_t, \phi) = ( w_t - v_t, \phi )$ yielding 
\begin{equation}
\label{preapriori}
\frac{dW}{dt} = (v_t , \psi'(v)  - \phi)   +  (w_t ,\psi'(w)  + \phi) + 
({\bf u}_t, { \bf u})
\end{equation}
The quantities $\psi'(v)  - \phi$ and $\psi'(w)  + \phi$ are called
the electro-chemical potential of $v$ and $w$ respectively. 
Note that $v_t = \nabla \cdot (v  \nabla \log(ve^{-\phi })) = \nabla \cdot ( v \nabla(\psi'(v) - \phi))$
Integrating by parts and using \eqref{F1} 
\begin{equation}
\label{calc1}
(v_t , \psi'(v)  - \phi) = -(v, |\nabla  \log(ve^{-\phi })|^2 ) - ({\bf u} \cdot \nabla v, \psi'(v) -\phi).
\end{equation}
Similarly, by \eqref{NP2} and \eqref{F1},
\begin{equation}
\label{calc2}
(w_t , \psi'(w)  + \phi) = -(w, |\nabla  \log(w e^{\phi })|^2 ) - ({\bf u} \cdot \nabla w,  \psi'(w)  + \phi).
\end{equation}
Since $\mathbf{u}$ is divergence free vector field which vanishes on the boundary of $\Omega,$ integration by parts gives
\begin{equation*}
0 =({\bf u} \cdot \nabla {\bf u}, {\bf u}) 
= (\nabla p,  {\bf u})
=({\bf u} \cdot \nabla v,  \psi'(v))
=  ({\bf u} \cdot \nabla w, \psi'(w)).
\end{equation*}
Similarly,  since $\phi$ is a solution of the Poisson equation \eqref{poisson}
$({\bf u} \cdot  \nabla(v - w), \phi)
= -({\bf u} \cdot \nabla \phi, \Delta \phi).$
From \eqref{NS}, we have then
\begin{equation}
\label{calc3}
\begin{split}
({\bf u_t}, {\bf u}) &= (\Delta \mathbf{u}, \mathbf{u}) -  (\nabla p,  {\bf u}) - ({\bf u} \cdot \nabla {\bf u}, {\bf u})
+ ({\bf u}\cdot \nabla \phi, \Delta \phi)\\
&=  -|\nabla {\bf u}|^2  + ({\bf u}\cdot \nabla \phi, \Delta \phi).
\end{split}
\end{equation}
Adding  \eqref{calc1}, \eqref{calc2} and \eqref{calc3} together
the following terms, which are interpreted as 
entropy production due to transport and the
kinetic energy production due to forcing, 
\begin{equation*}
(\mathbf{u}\cdot \nabla v, \phi), \quad
-(\mathbf{u}\cdot \nabla w, \phi),\quad
(\mathbf{u} \cdot \nabla \phi, \Delta \phi)
\end{equation*}
cancel.  
We find
\begin{equation}
\label{apriori}
\frac{d{W}}{dt} =- \int_{\Omega} v |\nabla \log (v e^{-\phi})|^2 + w |\nabla \log (w e^{\phi})|^2 +  |\nabla {\bf u}|^2 \,dx
 \leq0
\end{equation}
for  all $0 \leq t \leq T.$  
\begin{remark}
The identity \eqref{apriori} is the basic energy law for the hydrodynamic Debye-H\"uckel model.
It, along with \eqref{conserved}, implies that 
\begin{equation*}
\begin{aligned}
\|v(t), w(t)\|_{\mathbf{L}\log \mathbf{L}}
+ \|\nabla \phi(t)\|_{\mathbf{L}^2}^2 
+ \|\mathbf{u}(t)\|_{\mathbf{L}^2}^2 
+ \int_0^t \|\nabla \mathbf{u}\|_{\mathbf{L}^2}^2 \,ds 
\leq \\
\|v_0, w_0\|_{\mathbf{L}\log \mathbf{L}}
+ \|\nabla \phi_0\|_{\mathbf{L}^2}^2  
+ \|\mathbf{u}_0\|_{\mathbf{L}^2}^2, \quad t > 0.
\end{aligned}
\end{equation*}
Here $\phi_0$ is the solution of the Poisson equation
with right hand side $v_0-w_0.$ 
It will be crucial in estabilishing a uniform $\mathbf{L}^2$
estimate when $\mathrm{dim}\, \Omega  = 2,$
c.f. lemmas \ref{loglemma} and \ref{extends}. 
\end{remark}
\begin{remark}
If we assume that $\mathbf{u},v,w$ and $\phi$ are a classical solution
of \eqref{NS}-\eqref{I} and $\phi$ satifies the boundary condition
\begin{equation}
\label{phinat}
\frac{\partial \phi}{\partial \nu} = 0, \quad (x,t) \in \Omega \times (0,\infty)
\end{equation}
in place of \eqref{D} 
\emph{or} one replaces \eqref{S}, \eqref{F1}, and \eqref{D} with the assumption 
\begin{equation}
\label{torus}
\Omega = \mathbf{T}^n \quad (n\mbox{-dimensional flat torus}),
\end{equation}
then an additional energy law holds;
\begin{equation*}
\frac{d}{dt}\|v,w\|_{\mathbf{L}^p}^p + \frac{4(p-1)}{p}\|\nabla v, \nabla w\|_{\mathbf{L}^2}^2 \leq 0, \quad 1 < p < \infty.
\end{equation*}
One has equality in the above relation if and only if $v(x,t)$ and $w(x,t)$ are equal for
a.e. $x \in \Omega.$  
Assuming either \eqref{phinat} or \eqref{torus}, 
a necassary condition for a static solution (one where $v_t = w_t = 0, \mathbf{u}_t = 0$)
\footnote{In fact, the static solution is the stationary solution} 
to exist 
is 
\begin{equation*}
\int_{\Omega} v_0 \,\mathrm{d}x = \int_{\Omega} w_0 \,\mathrm{d}x.
\end{equation*}
In this case, the static solution will be 
\begin{equation*}
\phi = \mathrm{const}, \quad v = w = \mathrm{const}.
\end{equation*}
The appeal of electrolyte fluids in application is the presence of 
sharp boundary layers in the charges and potential when a static 
equilibrium is reached.  From the point of view of the analysis
and physicality of the model, the assumptions \eqref{phinat} or \eqref{torus}
lead to an over simplified model.  The boundary conditions 
used in the model in this paper are the physical ones but lead
to significant difficulties in the analysis.
\end{remark}

\subsection{Definitions}
The following function spaces will be used throughout this paper.
\begin{gather*}
{\bf H}^k({\Omega}) = {\bf W}^{k,2}({\Omega}) \mbox{ is the Sobolev space with norm } \|\cdot\|_{\mathbf{H}^k},\\
{\bf H}^{-1} = \mbox{ dual of } {\bf H}^1({\Omega}),\\
\mathscr{V}({\Omega}) = { \bf C}^{\infty}_0({\Omega}; \mathbf{R}^n) \cap \left\{ \mathbf{v}: \nabla \cdot \mathbf{v}  = 0\right\},\\
\mathbf{H}({\Omega}) = \mbox{ closure of } \mathscr{V} \mbox{ in } \mathbf{L}^2({\Omega}),\\
\mathbf{V}({\Omega}) = \mbox{ closure of } \mathscr{V} \mbox{ in } \mathbf{H}^1({\Omega}),\\
{\bf V}^{*} = \mbox{ dual of } {\bf V}({\Omega}), \quad
{\bf H}^{*} = \mbox{ dual of } {\bf H}({\Omega}).
\end{gather*}
The dependence on $\Omega$ will be omitted when the context is clear.
For $0 \leq t \leq T \leq \infty,$ define 
\begin{equation*}
\mathbf{Q}_T = \Omega \times (0,T), \quad \mathbf{Q}_{(t,T)} = \Omega \times (t,T).
\end{equation*}
For $k$ a positive integer and $0 < \alpha < 1,$ $\Omega \subset \mathbf{R}^n$ and 
$ U \subset \mathbf{R}^n \times  (0,\infty)$ open, the H\"older spaces 
\begin{equation*}
\mathbf{C}^{k + \alpha}(\Omega), \quad \mathbf{C}^{k + \alpha}(U)
\end{equation*}
are defined in chapter 3 of \cite{LIEBERMAN}.   
Note that $\mathbf{C}^{k+\alpha}(\Omega)$ is defined with respect to the Euclidean distance
while the  space $\mathbf{C}^{k + \alpha}(U)$ is defined with respect to the parabolic distance.
Recall
that if $k$ is an integer and $v \in  \mathbf{C}^{2k + \alpha}(U),$ $ U \subset \mathbf{R}^n \times  (0,\infty),$  
then $v$ has $2k$ uniformly H\"older
continuous derivatives in $x$ and $k$ uniformly H\"older
continuous derivatives in $t$ both with exponent $\alpha.$

\subsection{Weak Solutions}
Throughout, $p^*$ will denote the critical Sobolev exponent $\frac{2n}{n-2}.$ 
Let $T > 0, \mathbf{u} \in \mathbf{L}^2((0,T); \mathbf{L}^{q^*} \cap \mathbf{V})$ 
and $\nabla \phi \in \mathbf{L}^2((0,T); \mathbf{L}^{q^*})$ 
where $\frac{1}{q^*} + \frac{1}{p^*} = \frac{1}{2}.$ 
We say $v \in \mathbf{L}^2((0,T); \mathbf{H}^1)$ 
is a weak solution (c.f. \cite{LIEBERMAN}, Chapter 10, Section 6) 
of the equations
\begin{equation}
\label{genNP}
\left\{
\begin{aligned}
&v_t + \mathbf{u} \cdot \nabla v = \nabla \cdot (\nabla v - v \nabla \phi), && x \in \Omega, t > 0,\\
&\frac{\partial v}{\partial \nu}  - v \frac{\partial \phi }{\partial \nu}  = 0,
&&x \in \partial \Omega, t > 0,\\
&v(x,0) = v_0(x), && x \in \Omega
\end{aligned}
\right.
\end{equation} 
on $\mathbf{Q}_T$ prodived 
\begin{equation*}
\begin{aligned}
(v(t) - v_0, \omega) = -\int_0^t (\nabla v - v \nabla \phi - v \mathbf{u}, \nabla \omega), \,\mathrm{d}s,\\
\forall \omega \in \mathbf{H}^1, \mbox{ a.e. } 0 <  t < T.
\end{aligned}
\end{equation*}
Setting $\omega$ to be the constant function $1,$ we see that a weak solution satisfies
the conserved mass equation, \eqref{conserved}.
Note that if $v \in \mathbf{L}^{p^*}$ and 
$\nabla \phi, \mathbf{u} \in \mathbf{L}^{q^*}$ for a.e. $t\in (0,T),$ 
then $v\nabla \phi$ and $v \mathbf{u}$ lie in $\mathbf{L}^2$ for a.e. $t$ 
and hence the trilinear terms are 
well defined.  By virtue of being a weak solution, $v$ is absolutely
continuous from $[0,T]$ into $\mathbf{L}^2.$ 
The Sobolev space $\mathbf{H}^1$ is separable and by the Lebesgue differntiation theorem, 
$v_t \in \mathbf{L}^1((0,T); \mathbf{H}^{-1})$ exists.

Let $T > 0$ and $\mathbf{f} \in \mathbf{L}^1((0,T); \mathbf{L}^2).$
We say $\mathbf{u} \in \mathbf{L}^2((0,T); \mathbf{V})$ 
is a (the, incase $n = 2$) weak solution (c.f. \cite{TEMAM01}, Chapter 3)
of the equations 
\begin{equation}
\label{genNS}
\left\{
\begin{aligned}
&\begin{aligned}
&\mathbf{u}_t + \mathbf{u} \cdot \nabla \mathbf{u} + \nabla p = \Delta \mathbf{u} + \mathbf{f},\\
&\nabla \cdot \mathbf{u} = 0,
\end{aligned} &&x \in \Omega, t > 0
\\ 
&\mathbf{u} = 0, \quad x\in  \partial \Omega, t > 0,\\
&\mathbf{u}(x,0) = \mathbf{u}_0(x), \quad \nabla \cdot \mathbf{u}_0 = 0, &&x \in \Omega
\end{aligned}
\right.
\end{equation} 
on $\mathbf{Q}_T$ provied 
\begin{equation*}
\begin{aligned}
(\mathbf{u}(t) - \mathbf{u}_0, \mathbf{v})
=- \int_0^t b(\mathbf{u},\mathbf{u}, \mathbf{v})  
+ (\nabla \mathbf{u}, \nabla \mathbf{v}) 
+ (\mathbf{f}, \mathbf{v})\,\mathrm{d}s, \; \forall \mathbf{v} \in \mathbf{V},\\
\mbox{ a.e. } t \in (0,T).
\end{aligned}
\end{equation*}
We have defined 
\begin{equation*}
b(\mathbf{u},\mathbf{v}, \mathbf{w}) = 
\int_{\Omega} \mathbf{u} \cdot \nabla \mathbf{v} \cdot \mathbf{w} \,dx
\left( = \sum_{i,j=1}^n \int_{\Omega} \mathbf{u}_i (\mathbf{v}_j)_{x_i} \mathbf{w}_j \,dx\right).
\end{equation*}
If $n \leq 4,$ then $p^* \geq 4$ and $b(\cdot,\cdot,\cdot)$ is well defined
on $(\mathbf{H}^1)^3.$  If $\mathbf{u}$ is a weak solution then it is absolutely continuous
from $[0,T]$ into $\mathbf{L}^2.$
The space $\mathbf{V}$ is separable and by the Lebesgue differntiation theorem, 
$\mathbf{u}_t \in \mathbf{L}^1((0,T); \mathbf{V}^{*})$ exists. 

If the body force $\mathbf{f}$ takes the form $\Delta \phi \nabla \phi,$ 
then we will sometimes define
\begin{equation}
\label{specialforce}
(\mathbf{f}, \mathbf{v}) = -(\nabla \phi \otimes \nabla \phi, \nabla \mathbf{v}),
\quad \forall v \in \mathbf{V}.
\end{equation}
The motivation for this definition is easily seen from the identity
$\nabla \cdot ( \nabla \phi \otimes \nabla \phi) = \Delta \phi \nabla \phi 
+ \frac{1}{2}\nabla| \nabla \phi|^2$ and integration by parts.  
The relevant consequences are as follows.  
Assume $\phi$ solves the Poisson equation \eqref{poisson} and \eqref{D}.
If for $1 \leq p \leq \infty, n < q$ 
$v_t,w_t \in \mathbf{L}^p((0,T); \mathbf{H}^{-1})$
and $v,w \in \mathbf{L}^{\infty}((0,T); \mathbf{L}^q),$
then by the injections 
$\mathbf{H}^{-1} \hookrightarrow \mathbf{H}^{1}_0,$
$\mathbf{L}^{q} \hookrightarrow \mathbf{W}^{1,q}_0$
induced by the Poisson equation and the Sobolev embedding 
$\mathbf{W}^{1,q} \subset \mathbf{C}^{\alpha},$ 
\begin{equation}
\label{specialconsequence}
\begin{aligned}
\|\mathbf{f}_t\|_{\mathbf{V}^*} 
&\leq 2\|\nabla \phi \otimes \nabla \phi_t\|_{\mathbf{L}^2}     \\
&\leq 2\|\nabla \phi\|_{\mathbf{L}^{\infty}}\|\nabla\phi_t\|_{\mathbf{L}^2}\\
&\leq c(\Omega)\|v,w\|_{\mathbf{L}^q}\|v_t, w_t\|_{\mathbf{H}^{-1}}
\in \mathbf{L}^p(0,T).
\end{aligned}
\end{equation}

Let $T > 0$ and $\mathbf{u}_0 \in \mathbf{H}$ and $v_0,w_0 \in \mathbf{L}^2.$ 
We say $\langle \mathbf{u}, v, w, \phi \rangle$ are a \emph{weak solution} of the boundary value problem 
\eqref{NS}-\eqref{I} on $\mathbf{Q}_T$ 
provided 
$\mathbf{u} \in \mathbf{L}^2((0,T); \mathbf{V})$ and $v,w \in \mathbf{L}^2((0,T); \mathbf{H}^1),$
and for all $\mathbf{v} \in \mathbf{V},\omega \in \mathbf{H}^1, \eta \in \mathbf{H}_0^1$
and a.e.  $0 < t < T,$ 
\begin{gather}
(\mathbf{u}(t) - \mathbf{u}_0, \mathbf{v})
=- \int_0^t b(\mathbf{u},\mathbf{u}, \mathbf{v})  
+ (\nabla \mathbf{u}, \nabla \mathbf{v}) 
+ (\Delta \phi \nabla \phi, \mathbf{v})\,\mathrm{d}s,\\  
(v(t) - v_0, \omega) = -\int_0^t (\nabla v - v \nabla \phi - v \mathbf{u}, \nabla \omega), \,\mathrm{d}s,\\
(w(t) - w_0, \omega) = -\int_0^t (\nabla w + w \nabla \phi - w \mathbf{u}, \nabla \omega), \,\mathrm{d}s,\\
-(\nabla \phi(t) \cdot \nabla \eta) = (v(t) - w(t), \eta)
\end{gather}

A weak solution $\langle \mathbf{u}, v, w, \phi \rangle $ of \eqref{NS}-\eqref{I} 
is said to be \emph{global} if it is defined for all $T\in (0,\infty).$
Note that by virtue of the injections 
$\mathbf{H}^k \hookrightarrow 
\mathbf{H}^{k+2} \cap \mathbf{H}^1_0, $
from the Poisson equation and the Sobolev embeddings 
$\mathbf{H}^2 \subset \mathbf{W}^{1,p^*} \subset \mathbf{C}^{\alpha}$, if 
$v,w \in \mathbf{L}^2((0,T);\mathbf{H}^1),$
then $\Delta \phi \nabla \phi \in \mathbf{L}^1((0,T); \mathbf{L}^{2}).$
It makes sense to speak of the weak solution $\mathbf{u}$ of the 
Navier Stokes equations for then 
$\mathbf{f} = \Delta \phi \nabla \phi.$ 
Similarly, if $n \leq 4,$ then $\mathbf{u} \in \mathbf{L}^2((0,T);\mathbf{L}^{q^*}\cap \mathbf{V})$
and it makes sense to speak of a weak solution $v$ and $w.$ 

A weak solution is said to be \emph{classical} if it is possesses enough differentiablity to
satisfy the equation(s) continuously in the usual sense.

\subsection{Stationary Solutions}
One arrives at the stationary equations of \eqref{NS}-\eqref{D} by
setting all derivatives in $t$ to zero and setting $\mathbf{u} \equiv 0.$ They are
\begin{equation}
\label{stationaryeqns}
\left\{
\begin{aligned}
&
\begin{aligned}
&\nabla \cdot \left(\nabla v -v \nabla \phi \right) = 0,\\
&\nabla \cdot \left( \nabla w +w \nabla \phi \right) = 0,\\
&\Delta \phi = v - w, \quad \nabla p =  \Delta \phi \nabla \phi,
\end{aligned}
&& x \in  \Omega\\
&\frac{\partial v}{\partial \nu} -v \frac{\partial \phi}{\partial \nu} =
\frac{\partial w}{\partial \nu} +w \frac{\partial \phi}{\partial \nu} =
\phi = 0, && x \in \partial \Omega\\
&\int_{\Omega} v \,\mathrm{d}x = 
\int_{\Omega} v_0 \,\mathrm{d}x, \quad
\int_{\Omega} w \,\mathrm{d}x = 
\int_{\Omega} w_0 \,\mathrm{d}x.
\end{aligned}
\right.
\end{equation}
We say $v,w \in \mathbf{H}^1, \phi \in \mathbf{H}^1_0$ is a weak solution of \eqref{stationaryeqns} 
provided the equations are satisfied weakly in the usual sense.
\begin{theorem}
\label{PBThm}
Let $\Omega \subset \mathbf{R}^n$ be open, bounded with smooth boundary $\partial \Omega.$ 
Let $M$ and $N$ be positive constants.
Then there exists 
a unique $\phi \in \mathbf{H}^1_0$
satisfying
\begin{equation}
\label{PB}
\Delta \phi =  M 
\frac{e^{\phi}}{\int_{\Omega} e^{\phi} \,\mathrm{d}y} 
- N 
\frac{e^{-\phi}}{\int_{\Omega} e^{-\phi}\,\mathrm{d}y }.
\end{equation}
If $k$ is nonnegative, 
then $\phi \in \mathbf{C}^k(\Omega)$ and 
\begin{equation}
\label{smallmasslimit}
\lim_{M,N \rightarrow 0} \|\phi\|_{\mathbf{C}^k(\Omega)} = 0.
\end{equation}
\end{theorem}
\begin{proof}
The unique solution is found by 
means of the direct method of the calculus of variations.
For $\phi \in {\bf H}^1_0,$ define 
\begin{equation}
\label{CONVEX}
{J}[\phi] =  \frac{1}{2}\|\nabla \phi\|_{\mathbf{L}^2}^2  + 
M \log \int_{\Omega} e^{\phi} \,\mathrm{d}x +
N \log  \int_{\Omega} e^{-\phi} \,\mathrm{d}x .
\end{equation}
Equation \eqref{PB} is the Euler-Lagrange equation of $J[\cdot].$
Let $\phi, \psi \in \mathbf{L}^1,$ $\phi \neq \psi,$ and $0 < \lambda < 1.$ 
By H\"older's inequality,
\begin{equation*}
\begin{aligned}
\log \int_{\Omega} e^{\lambda \phi + (1-\lambda)\psi}\,\mathrm{d}x
&< \log\left\{ \left(\int_{\Omega} e^{\phi}\mathrm{d}x\right)^{\lambda}\cdot \left(\int_{\Omega} e^{\psi}\mathrm{d}x\right)^{(1-\lambda)}\right\}\\
&= \lambda \log \int_{\Omega} e^{\phi }\,\mathrm{d}x
+ (1-\lambda)\log \int_{\Omega} e^{\psi }\,\mathrm{d}x.
\end{aligned}
\end{equation*}
This shows that $J$ is strictly convex.  
By Jensen's inequality,
\begin{equation*}
J[\phi] \geq 
\frac{1}{2}\|\nabla \phi\|_{\mathbf{L}^2}^2
+ (M-N) \int_{\Omega} \phi \,\mathrm{d}x
+ (M+N)\log |\Omega|.
\end{equation*}
Then, by H\"older's inequality and the Poincar\'e inequality,
$J$ is bounded below by some constant depending only on $M,N$ and $\Omega.$ 
By the direct method of the calculus of variations, 
$J$ has a unique minimum $\phi \in \mathbf{H}^1_0.$ 
Since $J[\phi] < \infty,$ it follows that $e^{\phi}, e^{-\phi} \in \mathbf{L}^1$
and $\phi$ satisfies
\begin{equation*}
-\int_{\Omega} \nabla \phi \cdot \nabla \psi 
= M
\frac{\int_{\Omega} e^{\phi}\psi \,\mathrm{d} x}{\int_{\Omega} e^{\phi} \,\mathrm{d}y} 
- N 
\frac{\int_{\Omega} e^{-\phi}\psi \,\mathrm{d}x }{\int_{\Omega} e^{-\phi}\,\mathrm{d}y },
\quad \forall \psi \in \mathbf{C}^{\infty}_0.
\end{equation*}
By means of the identity, 
$a e^t - b e^{-t} = (ab)^{\frac{1}{2}}\sinh\{t - 
\frac{1}{2}(\frac{b}{a})\}$ $a, b > 0, t \in \mathbf{R},$
this equation is equivalent to 
\begin{equation}
\label{sinhweak}
-\int_{\Omega} \nabla \phi \cdot \nabla \psi 
= \alpha[\phi] \int_{\Omega} \sinh(\phi - \beta[\phi]) \psi \,\mathrm{d}x, 
\quad \forall \psi \in \mathbf{C}^{\infty}_0.
\end{equation}
where
$\alpha[u]$ and $\beta[u]$ are defined by
the relations
\begin{equation*}
\label{alphabetadef}
\alpha[\phi] = \left(\frac{MN}{\int_{\Omega} e^{\phi} \,\mathrm{d}y \int_{\Omega} e^{-\phi} \,\mathrm{d}y }\right)^{\frac{1}{2}},
\quad 
\beta[\phi] = \frac{1}{2} \log \left(\frac{N \int_{\Omega} e^{\phi} \,\mathrm{d}y}{M\int_{\Omega} e^{-\phi} \,\mathrm{d}y}\right).
\end{equation*}
Since $\sinh(\cdot)$ is increasing, applying the maximum principle to \eqref{sinhweak}
we find that $\phi \in \mathbf{L}^{\infty}.$ 
It follows from the standard theory for semilinear
elliptic equations (e.g. \cite{TAYLORIII}, Chapter 14), 
that $\phi \in \mathbf{C}^{k}(\Omega)$ for all $k > 0.$   

If $\tilde{\phi}\in \mathbf{H}^1_0$ and $J[\tilde{\phi}] < \infty,$
then 
\begin{equation*}
\frac{J[(1-\lambda)\tilde \phi + \lambda\phi] - J[\tilde \phi]}{\lambda}
\leq J[\phi] - J[\tilde{\phi}].
\end{equation*}
If $\tilde \phi \neq \phi,$ then the right hand side
is negative and $\tilde \phi$ is not a critical
point of $J.$  Hence any solution of \eqref{PB} is identically
$\phi.$

Assume without loss of generality that $M \leq N.$ 
Define
\begin{equation*}
\zeta = \Delta \phi.
\end{equation*}
One checks that $\zeta$ satisfies the equation
\begin{equation*}
\Delta \zeta = (|\nabla \phi|^2 + \eta)\zeta
\end{equation*}
where 
\begin{equation*}
\eta =  M
\frac{e^{\phi}}{\int_{\Omega} e^{\phi} \,\mathrm{d}y}
+ N
\frac{e^{-\phi}}{\int_{\Omega} e^{-\phi}\,\mathrm{d}y }.
\end{equation*}
The function $\eta$ is positive.  By the maximum principle, 
$\zeta$ has no negative internal minima nor positive internal maxima.
Since
\begin{equation*}
\int_{\Omega} \zeta \,\mathrm{d}x = M - N \leq 0,
\end{equation*}
we see that $\zeta$ is nonpositive.  Since $\phi$ 
restricted to $\partial \Omega$ is 0, 
\begin{equation}
\label{phineg}
\phi(x) \leq 0,\quad \forall x \in \Omega.
\end{equation}

Choosing the argument of \eqref{CONVEX} to be the constant
function $0,$ we see that 
\begin{equation*}
J[\phi] \leq (M+N)\log |\Omega|.
\end{equation*}
Arguing as in the beginning of the proof we find
that there is a constant $\clabel{smallmass}$ depending only on $\Omega$ for which
\begin{equation*}
\|\nabla \phi\|_{\mathbf{L}^2} \leq \cref{smallmass}(N-M)^2.
\end{equation*}
Applying Jensen's inequality and the Poincar\'e inequality
once more, 
\begin{equation}
\label{intnotsmall}
\int_{\Omega} e^{\phi}\,\mathrm{d}x
\geq |\Omega|e^{\clabel{smallmass2}(M-N)^2}
\end{equation}
for some constant $\cref{smallmass2}$ depending only on $\Omega.$ 

Combining \eqref{phineg} and \eqref{intnotsmall},
\begin{equation*}
\Delta \phi \leq \frac{M}{|\Omega|}e^{\cref{smallmass2}(N-M)^2},
\quad \phi(x) = 0, \quad x \in \partial \Omega.
\end{equation*}
By the maximum principle, $\phi$ converges to $0$ uniformly
in $\Omega$ as $M,N$ converge to $0.$   The estimate
\eqref{smallmasslimit} now follows by bootstrapping the 
elliptic estimates for $\phi.$
\end{proof}

\begin{corollary}
\label{statexist}
There exists a unique solution $v_{\infty}, w_{\infty},
\phi_{\infty}$ of the stationary equations \eqref{stationaryeqns}.
If $k > 0,$ then $v_{\infty},w_{\infty},\phi_{\infty} \in \mathbf{C}^k(\Omega)$
and 
\begin{gather*}
v_{\infty} = \int_{\Omega} v_0 \,\mathrm{d}x e^{-\phi_{\infty}},\quad
w_{\infty} = \int_{\Omega} w_0 \,\mathrm{d}x e^{\phi_{\infty}},\\
\lim_{\rho_0 \rightarrow 0} 
\left\|v_{\infty} \right\|_{\mathbf{C}^k(U)}=
\lim_{\rho_0 \rightarrow 0} 
\left\|w_{\infty} \right\|_{\mathbf{C}^k(U)} = 0.
\end{gather*}
\end{corollary}
\begin{proof}
Suppose $v,w\in \mathbf{H}^1, \phi \in \mathbf{H}^1_0$ is weak solution of \eqref{stationaryeqns}.
For $\delta > 0,$ consider the test function
\begin{equation*}
\psi = \log\{v_+e^{-\phi} + \delta\} \in \mathbf{H}^1.
\end{equation*}
Multiplying the first equation in \eqref{stationaryeqns} by $\psi$
and integrating by parts, 
\begin{equation*}
0 = \int_{\Omega} (\nabla v - v \nabla \phi)\cdot \nabla \psi \,\mathrm{d}x
= \int_{v > 0} \frac{e^{-\phi}|\nabla v - v \nabla \phi|^2}{ v_+e^{-\phi} + \delta} \,\mathrm{d}x.
\end{equation*}
This implies that $\nabla v = v \nabla \phi$ for a.e. $x \in \{ y \in \Omega : v(y) > 0\}.$
If
\begin{equation*}
\psi = \log\{(-v)_+e^{-\phi} + \delta\} \in \mathbf{H}^1,
\end{equation*}
then the above reasoning also implies that 
$\nabla v = v \nabla \phi$ for a.e. $x \in \{ y \in \Omega : v(y) < 0\}.$
Using the differentiability of Sobolev functions on lines (e.g. 
\cite{EvGa92}, Section 4.9), we find that 
\begin{equation*}
v(x) = \int_{\Omega} v_0 \,\mathrm{d}x e^{-\phi(x)}, \quad \mbox{ for a.e. } x \in \Omega.
\end{equation*}
Similarly 
\begin{equation*}
w(x) = \int_{\Omega} w_0 \,\mathrm{d}x e^{\phi(x)}, \quad \mbox{ for a.e. } x \in \Omega.
\end{equation*}
It follows that $e^{\phi},e^{-\phi} \in \mathbf{H}^1.$

Returning to Theorem \ref{PBThm}, $\phi$ is the unique solution 
of \eqref{PB} with $M = \int_{\Omega} v_0 \,\mathrm{d}x$
and $N = \int_{\Omega} v_0 \,\mathrm{d}x.$  It follows that 
$v$ and $w$ are, a fortiori, unique.
\end{proof}

\begin{remark}
Let $\sigma_{\infty} = \nabla \phi_{\infty} \otimes \nabla \phi_{\infty}
- \frac{1}{2}|\nabla \phi_{\infty}|^2I.$ Observe that 
Observe that 
\begin{equation}
\label{pressure}
\nabla \cdot \sigma_{\infty} = \Delta\phi_{\infty} \nabla \phi_{\infty} = \nabla (v_{\infty} + w_{\infty}).
\end{equation}
Equation \eqref{pressure} states that the divergence of the stationary 
electric stress $\sigma_{\infty}$ 
is the gradient of a pressure.
This is consistent with the fourth equation in \eqref{stationaryeqns}.
\end{remark}

\subsection{Main Results}
The following existence, uniqueness and regularity theorem 
is the strongest result expected from \eqref{NS}-\eqref{I}
for general data.
\begin{theorem}
\label{globalex}
If $\mathrm{dim}\, {\Omega}= 2,$ and 
${\bf u}_0 \in \mathbf{H}$ and $v_0,w_0 \in \mathbf{L}^2,$ 
then \eqref{NS}-\eqref{I}
possesses a unique, global weak solution.  
The solution is classical.  In particular, if $0 < t < T < \infty$ and
$\mathscr{R}$ is any compact subset of $\mathbf{Q}_{T},$ then
\begin{equation*}
\mathbf{u} \in \mathbf{C}^{2 + \alpha}(\mathscr{R})
\cap \mathbf{C}^{\alpha}(\mathbf{Q}_{(t,T)}), \quad 
v , w \in \mathbf{C}^{2 + \alpha}(\mathbf{Q}_{(t,T)}).
\end{equation*}
\end{theorem} 
Thus, in two space dimensions, \eqref{NS}-\eqref{I}
is solvable and enjoys usual regularization 
property found in equations of parabolic type. 
In dimensions three and four, 
the existence of a global weak solution can be proved
using the techniques in the proof of theorem \ref{globalex}
assuming a uniform in time $\mathbf{L}^2$ apriori estimates 
for $v$ and $w.$ 

The next theorems concern the long term behavior of
weak solutions.  In order to quantify the 
convergence, define for $1 \leq p \leq \infty,$ 
\begin{equation*}
\begin{aligned}
\mathscr{E}_p(t) &=\\ 
&\int_{\Omega} |{\bf u}(t)|^2 + 
\frac{|v(t) - v_{\infty}|^p}{v_{\infty}^{p-1}} + \frac{|w(t) - w_{\infty}|^p}{w_{\infty}^{p-1}} + 
|\nabla \phi(t) - \nabla \phi_{\infty}|^2 \,\mathrm{d}x
\end{aligned}
\end{equation*}
The following theorem is modeled after \cite{BiDo00}.
Note that, in contrast to theorem \ref{globalex}, these solutions may not be defined globally 
if $n \geq 3.$ 
\begin{theorem}
\label{l1ass}
Let ${\Omega} \subset \mathbf{R}^n,$ be bounded and uniformly convex, 
$n \geq 2,$ 
and $\langle \mathbf{u},v,w,\phi \rangle $ be a global weak solution of 
\eqref{NS}-\eqref{I}.
Then there is $\lambda_1 > 0$ depending on $\Omega$ and $e_1 < \infty$ depending 
only on the initial data   
so that for all $t \geq 0$ 
\begin{equation}
\label{l1assin}
\mathscr{E}_1(t) \leq e_1 e^{-\lambda_1 t}.
\end{equation} 
\end{theorem}

It is difficult to extend the technique of \cite{BiDo00}
used in the proof of theorem \ref{l1ass} to $\Omega$ with general geometry. 
However, if one assumes that the initial data is close to the
stationary solution and the stationary solution is small
then the convergence with a rate is recovered.
This result is a kind of linearization of the 
argument used in the proof of theorem \ref{l1ass}.
\begin{theorem}
\label{l2ass}
Let ${\Omega} \subset \mathbf{R}^n,$  $n = 2, 3,$ $\mathbf{u}_0 \in \mathbf{H},v_0,w_0 \in \mathbf{L}^2.$ 
There are positive constants  
\begin{equation*}
\rho_2 = \rho_2(\Omega), \quad \lambda_2 = \lambda_2(\Omega), \quad 
\epsilon_2 = \epsilon_2(\Omega )
\end{equation*}
such that if 
\begin{equation*}
\mathscr{E}_2(0) < \epsilon_2, \quad \rho_0 < \rho_2
\end{equation*}
then \eqref{NS}-\eqref{I}
possesses a global, weak solution
and 
\begin{equation}
\label{l2assin}
\mathscr{E}_2(t) \leq \epsilon_2 e^{-\lambda_2 t}.
\end{equation}
If a global weak solution satisfies 
\begin{equation*}
\sup_{t\in(0,\infty)} \|v,w\|_{\mathbf{L}^2} < \infty,
\end{equation*}
then there is $t_0 > 0$ so that $\mathscr{E}_2(t_0) < \epsilon_2.$ 
\end{theorem}
Thus, in two  space dimensions, theorem \ref{l2ass} implies that 
the solution from theorem \ref{globalex}  
tends to the stationary solution 
since the weak solution 
is eventually close to the stationary solution.  

In three dimensions, a global existence, uniqueness and regularity
result is proved under a small data assumption.  
However, it not sufficient to assume that the initial
data is small and close to the stationary solution.  
One must also assume that the stationary solution is also small.
\begin{theorem}
\label{globalex3d}
Let $\mathrm{dim}\, {\Omega}= 3,$ and 
${\bf u}_0 \in \mathbf{H}$ and $v_0,w_0 \in \mathbf{L}^2.$ 
There exist constants 
\begin{equation*}
\rho_3 = \rho_3(\Omega), \quad \epsilon_3 = \epsilon_3(\Omega), \quad 
\delta_3 = \delta_3(\Omega) 
\end{equation*}
such that if 
\begin{equation*}
\rho_0 < \rho_3, \quad \mathscr{E}_2(0) < \epsilon_3, 
\quad 
\|\mathbf{u}_0, v_0, w_0\|_{\mathbf{H}^2} < \delta_3,
\end{equation*}
then \eqref{NS}-\eqref{I} possesses a unique, global classical solution.
In particular, if $0 < t < T$ and 
$\mathscr{R}$ is any compact subset of $\mathbf{Q}_{T},$ then
\begin{equation*}
\mathbf{u} \in \mathbf{C}^{2 + \alpha}(\mathscr{R})
\cap \mathbf{C}^{\alpha}(\mathbf{Q}_{(t,T)}), \quad 
v , w \in \mathbf{C}^{2 + \alpha}(\mathbf{Q}_{(t,T)}).
\end{equation*}
\end{theorem}

The remainder of the paper is organized as follows.
In section \ref{preliminaries}, some sufficient
conditions are developed for concluding $\mathbf{L}^{\infty}$ 
bounds on weak solutions are developed.  These are later used 
proofs of theorems \ref{globalex} and \ref{globalex3d}.
As noted in the introduction, this is a necassary step 
for the regularity programme due to the lack 
of a gradient descent structure for the conservation equations
and the lack of a maximum principle for 
\eqref{NP1} and \eqref{NP2}. 
In section \ref{WEAKSOLUTIONS}, local weak solutions are 
constructed and the extension property is developed
in two space dimensions.  The section is concluded with
the proof of theorem \ref{globalex}.  
In section \ref{Dolbeaultsproof}, we give a proof of
theorem \ref{l1ass} modeled after the result 
of \cite{BiDo00}.  Finally, in section \ref{globalbehave}
we give several preparatory lemma and provide the 
proofs theorems \ref{l2ass} and  \ref{globalex3d}.

The letter $C$ will denote a constant which may change from
line to line within a proof.  The letters $c_1,c_2,\dots,$ 
will denote constants that are fixed throughout the paper.

\section{Acknowledgment}
This work was completed while the author
was Lovett Instructor in the Department of 
Mathematics at Rice University.  The author
would like to thank Chun Liu, Robert Hardt, Thierry De Pauw,
and Jean Dolbeault for guidance and many fruitful
discussions on the system presented in this paper.

\section{Preliminaries}
\label{preliminaries}
In this section we are assuming $\Omega$ is an open subset of 
$\mathbf{R}^2$ or $\mathbf{R}^3$ and that $\partial \Omega$ is smooth.
The following preliminary results will later be used to infer 
uniform $\mathbf{L}^{\infty}$ bounds on solutions $v,w, \mathbf{u}.$  
The $\mathbf{C}^{2 + \alpha}$ 
regularity of solutions will then follow from classical results
for linear second order PDE of parabolic type.

In the case of $\mathbf{L}^{\infty}$ bounds on weak solutions $v$ and $w$ of
\eqref{NP1} and \eqref{NP2}, the usual techniques for nondivergence form 
semilinear parabolic equations do not apply, mainly due to the boundary conditions \eqref{F1}.
Instead, we rely on a Moser type iteration
argument.  
The essential part of the argument is that 
due to the divergence free condition, 
no regularity on the velocity $\mathbf{u}$ need  
be assumed. 

First a general 
\begin{lemma}
\label{infdiff}
For $t \in [0,T]$ and $p \in [1,\infty),$ let $y(t,p)$ be positive and continuous and
 satisfy the differential inequality
\begin{equation*}
\frac{\partial y^p}{\partial t}(t,p) + y^p(t,mp) \leq w(t)p^k y^{\alpha(p)}(t,p)
\end{equation*}
where $m >  1,k \geq 0,$  $w(t) \geq 0$ is measurable with $\int_s^t w(r)\,dr \leq \gamma|t-s|^{\beta}$
for some constants $\gamma, \beta > 0$ and $\alpha(p) \in (0,1].$  If $0 < \epsilon < t$ 
and $1 \leq p_0,$ then there is 
$\clabel{moser} = \cref{moser}(w, k, p_0, \epsilon) < \infty$ for which
\begin{equation*}
\overline{\lim_{p \uparrow \infty}}\; y(t,p) \leq \cref{moser} \cdot y(t-\epsilon, p_0).
\end{equation*}
\end{lemma}
\begin{proof}
Without loss of generality, 
we may assume that $\alpha(p) = 1$ for every $p.$  
For otherwise, we may replace $y(p,t)$ by $\max\{1, y(p,t)\}.$ 
Let $0 < s < t \leq T.$ 
Using Gronwall's inequality,
\begin{equation}
\label{consequence}
y^p(t,p) + \int_s^t y^p(r,mp) \leq \exp\left(p^k\int_s^t w(r)\,dr \right) y^p(s,p)
\end{equation}
We will take advantage of the various powers in this inequality to 
infer some bounds.

Let $0 < \epsilon \leq t$ and $\sigma \geq 2.$ Define
\begin{equation*}
t_i = t - \epsilon \sigma^{-i}, \quad \delta_i = t_{i+1} - t_i, \quad p_i = m^ip_0, \quad i = 1, 2, \dots
\end{equation*}
From  \eqref{consequence} with $p = p_{i+1}, t = t_{i+1}$ and $s \in [t_i, t_{i+1}]$ we have (recall $mp_i = p_{i+1}$)
\begin{equation*}
y^{p_i}(t_{i+1}, p_{i+1}) \leq \exp\left(\frac{p_{i+1}^k}{m}\int_{t_i}^{t_{i+1}} w(r)\,dr \right) y^{p_i}(s,mp_{i}).
\end{equation*}
Integrating this expression with respect to $r = s\in[t_i, t_{i+1}],$
\begin{equation*}
\delta_i y^{p_i}(t_{i+1}, p_{i+1}) \leq \exp\left(\frac{p_{i+1}^k}{m}\int_{t_i}^{t_{i+1}} w(r)\,dr \right) \int_{t_i}^{t_{i+1}} y^{p_i}(r,mp_i) \,dr.
\end{equation*}
Using \eqref{consequence} once more with $s = t_i, t= t_{i+1}$ and $p = p_i$ to bound the second integral on the right hand side,
\begin{equation*}
\delta_i y^{p_i}(t_{i+1}, p_{i+1}) \leq \exp\left(\left(\frac{p_{i+1}^k}{m} + p_i^k\right)\int_{t_i}^{t_{i+1}}w(r)\,dr \right)   y^{p_i}(t_{i},p_i).
\end{equation*}
This in turn implies
\begin{equation*}
y(t_{i+1}, p_{i+1}) \leq \delta_i^{-\frac{1}{p_i}}\exp\left(\left(\frac{p_{i+1}^k}{mp_i} + p_i^{k-1}\right)\int_{t_i}^{t_{i+1}}w(r)\,dr \right)y(t_{i},p_i).
\end{equation*}
Note the integrand in the argument of the exponential is bounded above by
\begin{equation*}
R = \gamma p_0^{k-1}\left(m^{k(i+1) - i-1} + m^{i(k-1)}\right){\delta_i}^\beta.
\end{equation*}
Clearly $2^{-1}\epsilon \sigma^{-i} \leq \delta_i \leq \epsilon \sigma^{-i}.$  
Choosing $\sigma =  \max\{2, m^{\frac{2k}{\beta}}\},$
\begin{equation*}
R \leq C_1 m^{-ki}\end{equation*}
for some $C_1 = C_1(\epsilon, \gamma, \beta, p_0,m,k).$ 
Similarly,
\begin{equation*}
{\delta_i}^{-\frac{1}{p_i}} \leq C_2(\epsilon, \beta, p_0, m)^{\frac{i}{m^i}}.
\end{equation*}
It follows that for all $i = 1, 2 \dots,$ 
$y(t_{i+1}, p_{i+1}) \leq C_2^{\frac{i}{m^i}} e^{C_1 m^{-ki}}y(t_{i}, p_{i})$
and so by recursion  
\begin{equation*}
y(t_{i+1}, p_{i+1}) \leq \Phi_i y(t_0, p_0)
\end{equation*}
where $\Phi_i = \Pi_{j=1}^i C_2^{\frac{i}{m^i}} \exp(C_1m^{-ki}) .$  This product converges
and we set the limit to be $\cref{moser}.$  The conclusion now follows by varying $p_0$ and $\epsilon$
in a sufficiently small subset of $[0,T] \times [1,2]$ and using the continuity of $y.$ 
\end{proof}

\begin{proposition}
\label{generalvw}
Let $S > 0,$  $\mathrm{dim}\,\Omega = 2,3,$
\begin{equation*}
\mathbf{v} \in \mathbf{L}^2((0,S);\mathbf{V}),
\end{equation*}
and $v_0, w_0 \in \mathbf{L}^2.$

1. Then there is $0 < T_0 =  T_0(\Omega, \|v_0,w_0\|_{\mathbf{L}^2}) \leq S$ so that the problem  
\begin{gather*}
v_t + {\bf v} \cdot \nabla v  = \nabla \cdot \left(\nabla v -v \nabla \phi \right),\\
w_t + {\bf v} \cdot \nabla w = \nabla \cdot \left( \nabla w +w \nabla \phi \right),\\
 \Delta \phi = v - w,\\
 \frac{\partial v}{\partial \nu} -v \frac{\partial \phi}{\partial \nu} = 0,\quad 
\frac{\partial w}{\partial \nu} +w \frac{\partial \phi}{\partial \nu} = 0, \quad  \phi = 0,\quad 
\mbox{ on } \partial \Omega \times (0,\infty),\\
v(x,0) = v_0(x),\quad w(x,0) = w_0(x) \quad x \in \Omega.
\end{gather*} 
has a unique weak solution on $\mathbf{Q}_{T_0}.$ 
Moreover, if $p \geq 2,$ there is a constant 
$\clabel{genlem}  = \cref{genlem}(\Omega,  \sup_{t \in (0,T_0)} \|v,w\|_{\mathbf{L}^2},p)$ so that 
\begin{equation}
\label{firstLp}
\sup_{t \in [0,T]} \|v,w\|_{\mathbf{L}^p} \leq \|v_0, w_0\|_{\mathbf{L}^p} e^{\cref{genlem}T}.
\end{equation}

2. If  $p \geq 2$ 
and $0 < s < t \leq T_0,$  then
there is 
\begin{equation*}
\clabel{genlemm} = \cref{genlemm}(\Omega, \|v,w\|_{\mathbf{L}^2((0,T);\mathbf{H}^1)}, p, t - s) < \infty
\end{equation*} 
for which the weak solution $v,w$ satisfies
\begin{equation}
\label{Linfest}
\|v(t),w(t)\|_{\mathbf{L}^{\infty}}  \leq \cref{genlemm} \cdot \|v(s),w(s)\|_{\mathbf{L}^p}.
\end{equation}

3. If, additionally, 
\begin{equation}
\label{smoothness}
\begin{split}
& v_0, w_0 \in \mathbf{C}^{2 + \alpha}(\Omega), \quad 
0 < v_0(x), w_0(x) \quad \forall x\in \overline{\Omega},\\
&\; \frac{\partial v_0}{\partial \nu} -v_0 \frac{\partial \phi_0}{\partial \nu} = 0, \quad 
\frac{\partial w_0}{\partial \nu} +w_0 \frac{\partial \phi_0}{\partial \nu} = 0,  \quad
\mbox{ on } \partial \Omega,
\end{split}
\end{equation}
 and $\mathbf{v} \in \mathbf{C}^{\alpha}(\mathbf{Q}_{T_0}),$
then
\begin{equation*}
v,w \in \mathbf{C}^{2 + \alpha}(\mathbf{Q}_{T_0})
\end{equation*}
and $v$ and $w$ are positive on $\overline{\mathbf{Q}_{T_0}}.$
\end{proposition}

\begin{proof}
(Part 1) The existence and uniqueness of a weak solution is established by making 
slight modifications to the 
proof of \cite{BiHeNa94} , theorem 1 to account for the term $\mathbf{v} \cdot \nabla.$  
We omit the details here. 
Let $0< T_0 =  T_0(\Omega, \|v_0\|_{\mathbf{L}^2}, \|w_0\|_{\mathbf{L}^2})$ so that the weak solution
is defined on $\mathbf{Q}_{T_0}$ and  let
\begin{equation*}
M_0 = \sup_{t \in (0,T_0)} \|v,w\|_{\mathbf{L}^2}, \quad
M_1 = \int_0^T \|v,w\|_{\mathbf{H}^1} \,dt.
\end{equation*}

Multiply  the $v$-equation and the $w$-equation by $pv^{p-1}$ and $pw^{p-1}$ respectively and integrate over
$\Omega.$  Since
$p \nabla v v^{p-1} = \nabla v^{p}$ and $p \nabla w w^{p-1} = \nabla w^{p},$
using $\nabla \cdot \mathbf{v} = 0$ and $\mathbf{v} \in \mathbf{H}^1_0 = 0$ we have
\begin{equation*}
0  = \int_{\Omega} p {\bf v}\cdot \nabla v v^{p-1} \,dx = \int_{\Omega} p{\bf v}\cdot \nabla w w^{p-1} \,dx.
\end{equation*}
Integrating by parts gives
\begin{equation}
\label{pest}
\begin{split}
&\frac{d}{dt}\|v,w\|_{\mathbf{L}^p}^p +
\frac{4(p-1)}{p}\|\nabla v^{\frac{p}{2}}, \nabla w^{\frac{p}{2}}\|_{\mathbf{L}^2}^2 
\\
&= 2(p-1) \int_{\Omega} \nabla \phi \cdot (v^{\frac{p}{2}}  \nabla v^{\frac{p}{2}} + w^{\frac{p}{2}}  \nabla w^{\frac{p}{2}}) \,dx
= A.
\end{split}
\end{equation}
for a.e. $t \in (0,T_0).$  
Since $\phi$ solves the Poisson equation, 
by elliptic regularity and the Sobolev embedding $\mathbf{H}^1 \subset \mathbf{L}^6,$  
we have 
\begin{equation*}
\|\nabla \phi\|_{\mathbf{L}^6} \leq C(\Omega, M_0), \quad \mbox{ a.e. } t \in [0,T_0)
\end{equation*}
For $\epsilon > 0$ we have from the Sobolev inequality 
\begin{equation*}
\|v^{\frac{p}{2}},w^{\frac{p}{2}}\|_{\mathbf{L}^3} 
\leq C(\Omega)(\epsilon^{-1} \|v, w\|_{\mathbf{L^p}}^{\frac{p}{2}} + \epsilon 
\|\nabla v^{\frac{p}{2}},\nabla w^{\frac{p}{2}}\|_{\mathbf{L}^2}).
\end{equation*}
Applying these inequalities, the estimate 
\begin{equation*}
\begin{split}
A &\leq 2(p-1)\|\nabla \phi\|_{\mathbf{L}^6} 
\|v^{\frac{p}{2}},w^{\frac{p}{2}}\|_{\mathbf{L}^3} \|\nabla v^{\frac{p}{2}}, \nabla w^{\frac{p}{2}}\|_{\mathbf{L}^2}\\
&\leq C(\Omega, M_0, p)(\epsilon^{-1} \|v, w\|_{\mathbf{L}^p}^{\frac{p}{2}} + \epsilon \|\nabla v^{\frac{p}{2}},\nabla w^{\frac{p}{2}}\|_{\mathbf{L}^2})
\|\nabla v^{\frac{p}{2}}, \nabla w^{\frac{p}{2}}\|_{\mathbf{L}^2}
\end{split}
\end{equation*}
follows.
Choosing $\epsilon \leq  (2C(\Omega, M_0, p))^{-1},$ we find from \eqref{pest} with $p \geq 2$
\begin{equation*}
\frac{d}{dt} \|v,w\|_{\mathbf{L}^p}^p \leq C(\Omega, M_0, p)\|v,w\|_{\mathbf{L}^p}^p.
\end{equation*}
Gronwall's inequality now implies \eqref{firstLp}.

\noindent(Part 2)
Returning to  \eqref{pest},  the inequality 
\begin{equation*}
\frac{d}{dt}\|v,w\|_{\mathbf{L}^q}^q 
+ \|\nabla v^{q/2}, \nabla w^{q/2}\|_{\mathbf{L}^2}^2 \leq 4q^2 
\|\nabla \phi\|_{\mathbf{L}^{\infty}}\|v,w\|_{\mathbf{L}^q}^q
\end{equation*}
holds for $q \geq 2.$ 
By elliptic regularity and the Sobolev embedding $\mathbf{W}^{1,4} \subset \mathbf{L}^{\infty},$
we have $\|\nabla \phi\|_{\mathbf{L}^{\infty}} \leq C(\Omega) \|v,w\|_{\mathbf{H}^1} \in \mathbf{L}^2(0,T_0).$
With this constant, define $w(t) = 4C(\Omega) \|v(t),w(t)\|_{\mathbf{H}^1}$
and note that 
\begin{equation*}
\int_s^t w(r) \,dr \leq \gamma|t-s|^{\frac{1}{2}}
\end{equation*}
where $\gamma = \gamma(\Omega, M_1).$ 
Furthermore, by the embedding 
${\bf H}^1 \subset {\bf L}^{4},$  
\begin{equation*}
\|v,w\|_{\mathbf{L}^{2q}}^q 
= \|v^{\frac{q}{2}},w^{\frac{q}{2}}\|_{\mathbf{L}^4}^2\leq 
C(\Omega)(\|\nabla v^{q/2}, \nabla w^{q/2}\|_{\mathbf{L}^2}^2 
+ \|v,w\|_{\mathbf{L}^q}^q).
\end{equation*}
Combining this with the previous observations, one has 
\begin{equation*}
\frac{d}{dt}\|v,w\|_{\mathbf{L}^q}^q + \|v,w\|_{\mathbf{L}^{2q}}^q \leq q^2 w(t) \|v,w\|_{\mathbf{L}^q}^q.
\end{equation*}
We may thus apply lemma \ref{infdiff} 
with $y(t,q) = \|v,w\|_{\mathbf{L}^q}, \alpha \equiv 1, \beta = \frac{1}{2}, p = k = m = 2$
to find \eqref{Linfest}.

\noindent (Part 3) If $v_0, w_0 \in \mathbf{C}^{2+ \alpha}(\Omega),$ then certainly $v_0, w_0 \in \mathbf{L}^{p}$
for $p > n.$  Fix $n < p \leq p^*.$ Then elliptic regularity and the Sobolev embedding $\mathbf{W}^{1,p} \subset 
\mathbf{L}^{\infty}$ and  \eqref{firstLp} imply there is $C = C(\Omega, M_0, T_0, p) < \infty$ so that 
\begin{equation*}
\|\nabla \phi\|_{\mathbf{L}^{\infty}(\mathbf{Q}_{T_0})} \leq C, \quad \forall t \in [0,T_0).
\end{equation*}
 If $\mathbf{v} \in \mathbf{C}^{\alpha}(\mathbf{Q}_{T_0}),$ then
$v$ and $w$ solve divergence form equations with oblique boundary conditions with bounded,
measurable coefficients.  By \cite{LIEBERMAN}, theorem 6.41 and 6.44, 
$v,w \in \mathbf{C}^{\alpha}(\mathbf{Q}_{T_0}).$  The relation $\Delta \phi = v - w$
then implies $\nabla \phi \in \mathbf{C}^{\alpha}(\mathbf{Q}_{T_0})$ as well.  
The $\mathbf{C}^{2 + \alpha}(\mathbf{Q}_{T_0})$  regularity of $v,w$ is then guaranteed by
the smoothness and consistency assumption \eqref{smoothness} and \cite{LIEBERMAN}, theorem 5.18.

One readily checks that $\|\min\{0,v ,w\}\|_{\mathbf{L}^2} = 0$ for all $t \in [0,T_0].$
If $v$ is not bounded below by a positive constant on $\mathbf{Q}_{T_0},$
then $v(x,t) = 0$ for some $(x,t) \in \mathbf{Q}_{T_0} \cup (\partial \Omega \times (0,T_0)).$
If $(x,t) \in \mathbf{Q}_{T_0},$ then the strong maximum principle (\cite{LIEBERMAN}, theorem 2.9)
implies $v \equiv 0,$ contradicting \eqref{posmass}.
If $(x,t) \in \partial \Omega \times (0,T_0),$ 
then by the parabolic Hopf lemma (\cite{LIEBERMAN}, theorem 2.6)    implies
\begin{equation*}
0 > \frac{\partial v}{\partial \nu}(x,t) = \frac{\partial v}{\partial \nu}(x,t) - v(x,t)  \frac{\partial \phi}{\partial \nu}(x,t) =0,
\end{equation*}
again a contradiction.  Analogous arguments apply to $w.$
\end{proof}

Solutions of the velocity equations
will be constructed from a special basis.
To describe the method, 
we choose a particular orthonormal basis 
$\{\zeta_i\}_{i=1}^{\infty}$ of $\mathbf{H}$ satisfying
\begin{gather*}
\Delta \zeta_i + \nabla p_i = -\lambda \zeta_i\\
\nabla \cdot \zeta_i = 0\\
\zeta_i(x) = 0 \quad \mbox{ for } x \in \partial \Omega.
\end{gather*}
Here $\zeta_i \in {\bf C}^{\infty}_0(\Omega)$ and 
$0 < \lambda_1 \leq \lambda_2 \leq \dots$ 
are eigen-pairs of the Stokes operator on ${\Omega}.$ 
For $i,j= 1,2, \dots,$ the functions $\zeta_i$ satisfy the orthogonality relations  
(c.f. section 2.6, \cite{TEMAM01})
\begin{equation}
 \label{ortho}
(\Delta \zeta_i \cdot \zeta_j) = - \lambda_i(\zeta_i, \zeta_j) = -\lambda_i \delta_{ij}.
 \end{equation}

Following proposition \ref{generalvw}, we are 
also interested in $\mathbf{L}^{\infty}$ 
bounds of the velocity $\mathbf{u}.$
The following general result for the Navier-Stokes equations will be useful
in the regularity proof in two space dimesions.

\begin{lemma}[Ladyzhenskaya's Inequality, c.f. \cite{TEMAM01}, Ch. 3]
\label{ladyzhenskaya}
Let $\Omega$ be an open subset of $\mathbf{R}^2$ or $\mathbf{R}^3.$ 
If $\mathbf{u} \in \mathbf{V}$ then 
\begin{gather*}
\|\mathbf{u}\|_{\mathbf{L}^4} 
\leq \|\mathbf{u}\|_{\mathbf{L}^2}^{\frac{1}{2}}
\|\nabla \mathbf{u}\|_{\mathbf{L}^2}^{\frac{1}{2}}, \quad \mathrm{dim}\, \Omega = 2,\\
\|\mathbf{u}\|_{\mathbf{L}^4} \leq
\|\mathbf{u}\|_{\mathbf{L}^2}^{\frac{1}{4}}
\|\nabla \mathbf{u}\|_{\mathbf{L}^2}^{\frac{3}{4}}, \quad \mathrm{dim}\, \Omega = 3.
\end{gather*}
\end{lemma}

\begin{proposition}
\label{uniformNS}
Suppose that $T > 0, \mathrm{dim}\; \Omega= 2, \mathbf{u}_0 \in \mathbf{L}^2,$ 
\begin{equation*}
\mathbf{f} \in \mathbf{L}^2((0,T); \mathbf{L}^2), \quad
t^{2}\mathbf{f}_t \in \mathbf{L}^2((0,T); \mathbf{V}^*),
\end{equation*} 
and $\mathbf{u}$ is
the unique Leray-Hopf solution of 
\begin{gather*}
\frac{\partial \mathbf{u}}{\partial t} + \mathbf{u} \cdot \nabla \mathbf{u}
+ \nabla p = \Delta \mathbf{u} + \mathbf{f}\\
\mathrm{div} \, \mathbf{u} = 0
\end{gather*}
on $\mathbf{Q}_{T}.$
Then there is a constant $\clabel{vellem} = \cref{vellem}(\Omega, \mathbf{f}, \mathbf{f}_t, \mathbf{u}_0) < \infty$ so that  
\begin{equation*}
\|\mathbf{u}_t\|_{\mathbf{L}^2}^2 \leq \frac{\cref{vellem}}{t^2},\quad
\|\nabla \mathbf{u}\|_{\mathbf{L}^2}^2 \leq \frac{\cref{vellem}}{t},\quad
\int_s^t \|\nabla \mathbf{u}_t\|_{\mathbf{L}^2}^2 \,dr 
\leq \frac{\cref{vellem}}{s^{2}}.
\end{equation*}
for all $0 < s < t \leq T.$
\end{proposition}
\begin{remark}
In proposition \ref{uniformNS} as in theorem \ref{globalex},
we are merely assuming $\mathbf{u}_0 \in \mathbf{L}^2.$
Also, $\mathbf{f}_t$ may be singular at the origin.  
We conclude that $\mathbf{u}$ is $\mathbf{C}^{\alpha}$
away from the set $\Omega \times \{t= 0\}.$
\end{remark}
\begin{proof}
For the moment, assume 
$\mathbf{f}_t \in \mathbf{C}^0([0,T]; \mathbf{L}^2).$
Let $m$ be a positive integer and 
$\mathbf{v}(x,t) = \sum_{i=1}^m \zeta_i(x) v_i(t)$ and 
$v_i$ solve $m$-dimensional system of ordinary differential equations
\begin{equation}
\label{ode1}
\dot{v}_i + \lambda_i v_i + \sum_{j,k=1}^m B_{ijk}v_j v_k = F_i(t).
\end{equation}
We have defined 
\begin{equation*}
B_{ijk} =b(\zeta_j, \zeta_k,\zeta_i)
\quad
F_i(t) =(\mathbf{f}, \zeta_i).
\end{equation*}
It is well known (c.f. \cite{TEMAM01}, Chapter 3) that 
for some subsequence of $m,$ the $\mathbf{v}$ converge 
to $\mathbf{u}$ 
in the strong topology $\mathbf{L}^2(\mathbf{Q}_T)$
and there is $c_0$ depending only on $\mathbf{u}_0$ and 
$\mathbf{f}$ for which 
\begin{equation}
\label{l2decay}
\|\mathbf{v}\|_{\mathbf{L}^{\infty}((0,T); \mathbf{L}^2)}^2 
+  \|\nabla \mathbf{v}\|_{\mathbf{L}^{2}((0,T); \mathbf{L}^2)}^2 \leq c_0,
\quad \forall m > 0.
\end{equation} 
Differentiate \eqref{ode1} with respect to $t$
and multiply the resulting system component-wise
by $\dot{v}_i.$  We find that 
\begin{equation*}
(\mathbf{v}_{tt}, \mathbf{v}_t) 
+ b(\mathbf{v}_t, \mathbf{v}, \mathbf{v}_t)
= -(\Delta \mathbf{v}_t, \mathbf{v}_t) + (\mathbf{f}_t, \mathbf{v}_t)
\end{equation*} 
Multiply this equation by $t^2$ and define $\mathbf{w} = t\mathbf{v}_t.$
One easily checks that 
\begin{equation*}
\frac{1}{2}\frac{d}{dt}\| \mathbf{w}\|_{\mathbf{L}^2}^2
+ \|\nabla \mathbf{w}\|_{\mathbf{L}^2}^2
= b(\mathbf{w}, \mathbf{u}, \mathbf{w}) + t\|\mathbf{v}_t\|_{\mathbf{L}^2}^2
+ t(\mathbf{f}_t, \mathbf{w}).
\end{equation*}
On the other hand, if we multiply \eqref{ode}
by $\dot{v}_i$ componentwise, we find
\begin{equation*}
\|\mathbf{v}_t\|_{\mathbf{L}^2}^2
= (\Delta \mathbf{v}, \mathbf{v}_t)
- b(\mathbf{v}, \mathbf{v}, \mathbf{v}_t) 
+ (\mathbf{f}, \mathbf{v}_t).
\end{equation*}
Inserting this expression into the previous equation and noticing that
\linebreak $t(\Delta \mathbf{v}, \mathbf{v}_t) =
-t(\nabla \mathbf{v}, \nabla \mathbf{v}_t) =
- t\frac{1}{2}\frac{d}{dt}(\nabla \mathbf{v}, \nabla \mathbf{v})
+\frac{1}{2} (\nabla \mathbf{v}, \nabla \mathbf{v}),$ 
gives
\begin{equation*}
\frac{1}{2}\frac{d}{dt}
\left(\|\mathbf{w}\|_{\mathbf{L}^2}^2
+ t\|\nabla \mathbf{v}\|_{\mathbf{L}^2}^2
\right)
+ \|\nabla \mathbf{w}\|_{\mathbf{L}^2}^2
= \mathrm{A}(t) + \mathrm{B}(t) + \mathrm{C}(t)
\end{equation*} 
where 
\begin{gather*}
\mathrm{A}(t) = \frac{1}{2}\|\nabla \mathbf{v}\|_{\mathbf{L}^2}^2,\\
\mathrm{B}(t) = \mathrm{B}_1(t) + \mathrm{B}_2(t) = 
b(\mathbf{w}, \mathbf{v}, \mathbf{w}) -b(\mathbf{v}, \mathbf{v}, \mathbf{w}), \\
\mathrm{C}(t) = t(\mathbf{f}_t, \mathbf{w}) + (\mathbf{f}, \mathbf{w}).
\end{gather*}

From \eqref{l2decay}, is clear that $\mathrm{A}(t)$ is integrable
with integral bounded above by $\frac{1}{2}c_0$ along any 
measurable subset of $[0,T].$   Similary, one has 
\begin{equation*}
\mathrm{C}(t) \leq 
2t^2\|\mathbf{f}_t\|_{\mathbf{V}^*}^2 + 
C(\Omega) \|\mathbf{f}\|_{\mathbf{L}^2}^2
+ \frac{1}{4} \|\nabla \mathbf{w}\|_{\mathbf{L}^2}^2 
\end{equation*}
with the two left hand terms on the right hand side of the ineqaulity also integrable.
By H\"older's inequality,
\begin{equation*}
\mathrm{B}_1(t) 
= -\int_{\Omega} \mathbf{w} \cdot \nabla \mathbf{w} \cdot \mathbf{v}\,dx  
\leq 
\|\mathbf{w}\|_{\mathbf{L}^4}
\|\mathbf{v}\|_{\mathbf{L}^4}
\|\nabla \mathbf{w}\|_{\mathbf{L}^2}.
\end{equation*}
By Ladyzhenskaya's inequality,
$\|\mathbf{v}\|_{\mathbf{L}^4} \leq  
2^{\frac{1}{2}}\|\mathbf{v}\|_{\mathbf{L}^2}^{\frac{1}{2}}\|\nabla \mathbf{v}\|_{\mathbf{L}^2}^{\frac{1}{2}}$
which lies in $\mathbf{L}^4(0,T).$
Similarly, 
$\|\mathbf{w}\|_{\mathbf{L}^4}\|\mathbf{w}\|_{\mathbf{L}^2}  
\leq 2^{\frac{1}{2}}
\|\mathbf{w}\|_{\mathbf{L}^2}^{\frac{1}{2}}
\|\nabla \mathbf{w}\|_{\mathbf{L}^2}^{\frac{3}{2}}.$  Taking these elements into consideration,
we find
\begin{equation*}
\mathrm{B}_1(t) \leq 
8\|\mathbf{w}\|_{\mathbf{L}^2}^{2}\|\mathbf{v}\|_{\mathbf{L}^4}^{4}
+ \frac{1}{4}\|\nabla \mathbf{w}\|_{\mathbf{L}^2}^2.
\end{equation*} 
Similar ideas imply 
\begin{equation*}
\mathrm{B}_2(t) \leq 
2\|\mathbf{v}\|_{\mathbf{L}^4}^{4}
+ \frac{1}{4}\|\nabla \mathbf{w}\|_{\mathbf{L}^2}^2.
\end{equation*}
In total, we find that 
\begin{equation*}
\frac{d}{dt}
\mathrm{F}(t)
+ G(t)
\leq 
\mathrm{D}(t) + \mathrm{E}(t)\mathrm{F}(t)
\end{equation*}
where
\begin{gather*}
\mathrm{D}(t) = 
2t^2\|\mathbf{f}_t\|_{\mathbf{V}^*}^2 + 
2\|\mathbf{f}\|_{\mathbf{V}^*}^2
+4\|\mathbf{v}\|_{\mathbf{L}^4}^{4} \in \mathbf{L}^1(0,T),\\
\mathrm{E}(t) = 16\|\mathbf{v}\|_{\mathbf{L}^4}^{4} \in \mathbf{L}^1(0,T),\\
\mathrm{F}(t) = \|\mathbf{w}\|_{\mathbf{L}^2}^2
+ t\|\nabla \mathbf{v}\|_{\mathbf{L}^2}^2,\\
G(t) =  \frac{1}{2}\|\nabla \mathbf{w}\|_{\mathbf{L}^2}^2.
\end{gather*}
We have assumed that 
$\mathbf{f}_t$ is continuous from $[0,T]$ into $\mathbf{L}^2.$
Therefore, $\mathbf{v}$ and $\mathbf{v}_t$ remain
bounded in $\mathbf{H}^1$ and $\mathrm{F}(t) = \|\mathbf{w}\|_{\mathbf{L}^2}^2
+ t\|\nabla \mathbf{v}\|_{\mathbf{L}^2}^2 = 0$
when $t=0.$ 
By Gronwall's inequality, there is a constant $\cref{vellem}$
 independent of $m$ and the continuity of $\mathbf{f}_t$
 for which
\begin{equation*}
\mathrm{F}(t) 
+  \int_s^t G(r) \,dr
\leq \cref{vellem}, \quad \forall 0 \leq s \leq t \leq T.
\end{equation*}
The proposition now follows 
by first approximating $\mathbf{f}$ by functions continuously differentiable in $t,$ 
and then letting $m$ diverge along the given subsequence; 
for a.e. $t \in (0,T),$  
Fatou's lemma implies
\begin{equation*}
\begin{aligned}
 t^2 \|\mathbf{u}_t\|_{\mathbf{L}^2}^2
+ t\|\nabla \mathbf{u}\|_{\mathbf{L}^2}^2
+ t^2 \int_s^t \|\nabla \mathbf{u_t}\|_{\mathbf{L}^2}^2 \,dr\\
\leq \underline{ \lim}_m\left( F(t) + \int_s^t G(r)\,dr\right)  \leq \cref{vellem}.
\end{aligned}
\end{equation*}
\end{proof}

\section{The Local Existence and Extension Property in 2D.}
\label{WEAKSOLUTIONS}
We will construct weak solutions of \eqref{NS}-\eqref{I}
as the limit of modified Galerkin approximation. 
We look for solutions of \eqref{NS}-\eqref{poisson} with the form
\begin{equation*}
\mathbf{u}_{{m}}(x,t)= \sum_{i=1}^{{m}} u_i(t)\zeta_i(x).
\end{equation*}
The orthogonality of the $\zeta_i$ lead us
to the following approximation problem.
For
\begin{equation*}
F_i(t) = \int_{\Omega} \Delta \phi_m \nabla \phi_m \cdot  \zeta_i\,dx,
\end{equation*}
 consider a solution of 
\begin{gather}
\label{ode}
\dot{u}_i + \lambda_i u_i + \sum_{j,k=1}^m B_{ijk} u_j u_k = F_i(t)\\
\label{odeNP1}
\frac{\partial v_m}{\partial t} + {\bf u}_m \cdot \nabla v_m  = \nabla \cdot \left(\nabla v_m -v_m \nabla \phi_m \right),\\
\label{odeNP2}
\frac{\partial w_m}{\partial t} + {\bf u}_m \cdot \nabla w_m = \nabla \cdot \left( \nabla w_m +w_m \nabla \phi_m \right),\\
\label{odepoisson}
 \Delta \phi_m = v_m - w_m,\\
\label{odeI}
{v}_i(0) = \int_{\Omega} u_0 \cdot \zeta_i\,dx,\quad v_m(0,\cdot ) = v_0(\cdot), \quad w_m(0,\cdot) = w_0(\cdot).
\end{gather}
We will prove
\begin{theorem}
\label{mglobal}
Suppose $\mathrm{dim} \, \Omega = 2,$ $\mathbf{u_0} \in \mathbf{H}$
and  \eqref{smoothness} hold.
For any $m > 0, T \in (0,\infty),$ the problem
\eqref{ode}-\eqref{odeI} has a unique, classical  solution with
$D_t^k D_x^l \mathbf{u}_m \in \mathbf{C}^{\alpha}(\mathbf{Q}_{T})$ for $k = 1,2$ and
$l = 1, 2, \dots,$ and  $v_m,w_m \in \mathbf{C}^{2+\alpha}(\mathbf{Q}_{T}).$
\end{theorem}
Theorem \ref{mglobal} will be a consequence of the following lemmas.

\begin{lemma}
\label{smalltime}
Suppose $\mathrm{dim} \, \Omega = 2,3,$ $\mathbf{u}_0 \in \mathbf{H}$ and assume  \eqref{smoothness}.
There is $T_0 = T_0(\Omega, \|\mathbf{u}_0\|_{\mathbf{L}^2}, 
\|v_0\|_{\mathbf{L}^2},  \|w_0\|_{\mathbf{L}^2}) > 0$ such that the problem
\eqref{ode}-\eqref{odeI} has a unique, classical  solution with
$D_t^k D_x^l \mathbf{u}_m \in \mathbf{C}^{\alpha}(\mathbf{Q}_{T_0})$ for $k = 1,2$ and
$l = 1, 2, \dots,$ and  $v_m,w_m \in \mathbf{C}^{2+\alpha}(\mathbf{Q}_{T_0}).$
\end{lemma}
\begin{proof}
Let $s > 0$ and 
let $\{v_i\}_{i=1}^m \in \mathbf{W}^{1,2}((0,s); \mathbf{R}^m).$ Define 
$\mathbf{v}(x,t) =  \sum_{i=1}^m \zeta_i(x)v_i(t).$   
The embedding $\mathbf{W}^{1,2}(0,s) \subset \mathbf{C}^{\frac{1}{2}}(0,s)$ 
implies \linebreak $\mathbf{v} \in \mathbf{Lip}(\mathbf{Q}_s)$ (with respect to the parabolic distance.)
From proposition \ref{generalvw}, there exists $T_0 = T_0(\Omega, \|v_0\|_{\mathbf{L}^2},  \|w_0\|_{\mathbf{L}^2})$ and 
a solution $v,w\in \mathbf{C}^{2+\alpha}(\mathbf{Q}_{T_0})$  of the problem in proposition \ref{generalvw}.
Let $M_0 = \sup_{t \in (0,T_0)} \|v_0,w_0\|_{\mathbf{L}^2}.$
Define $F_i(t) = \int_{\Omega} \Delta \phi \nabla \phi \cdot \zeta_i \,dx.$ Then, by elliptic regularity
and \eqref{firstLp}
\begin{equation*}
\sup_{t \in [0,T_0]}\sum_{i=1}^m F_i^2(t) 
\leq  \sup_{t \in [0,T_0]} \|\Delta \phi \nabla \phi\|_{\mathbf{L}^1} \sum_{i=1}^m \|\zeta_i\|_{\mathbf{L}^\infty} \leq   C(m, \Omega, M_0).
\end{equation*}
Substituting $\Delta \phi \nabla \phi$ into \eqref{ode}, this ordinary differential equation has a 
\linebreak $\mathbf{W}^{1,\infty}(0,T_2)$ solution $\{\bar v_i\}_{i=1}^m.$   Infact, 
by \eqref{odeI} and \eqref{ortho}, we have \linebreak $\sum_{i=1}^m \bar{g}_i^2 (0) \leq |\mathbf{u}_0|_2^2.$
Therefore, the $\{\bar{v}_i\}_{i=1}^m$ satisfy the estimate
\begin{equation*}
\sum_{i=1}^m \bar{v}_i^2 (t) \leq  \|\mathbf{u}_0\|_{\mathbf{L}^2}^2 + t C(m, \Omega, M_0), \quad \forall t \in [0,T_0].
\end{equation*}
Fix $s = T_0$ and $M  = \|\mathbf{u}_0\|_{\mathbf{L}^2}^2 + T_0 C(m, \Omega, M_0).$  This construction
maps $\mathbf{W}^{1,2}((0,T_0); \mathbf{R}^m)$ into
\begin{equation*}
\mathbf{S}_M =  \left\{\{v_i\}_{i=1}^m \in \mathbf{W}^{1,2}(0,T_0; \mathbf{R}^m): \linebreak \sup_{t \in [0,T_0]} \sum_{i=1}^m v_i(t)^2 \leq M\right\}.
\end{equation*}
The mapping is clearly continuous and takes $\mathbf{S}_M,$ a compact and convex subset 
of $\mathbf{W}^{1,2}((0,T_0);\mathbf{R}^m),$ into itself.
A fixed point $\{v_i\}_{i=1}^m$ is thus guaranteed by the Schauder fixed point theorem (\cite{LIEBERMAN}, theorem 8.1.) Let $\{u_i\}_{i=1}^m$ be such a fixed point.
Let $v_m, w_m,\phi_m$ be the solution associated with $\mathbf{v} = \sum_{i=1}^m \zeta_i u_i.$ from this construction.
Define $\mathbf{u}_m = \mathbf{v} = \sum_{i=1}^m \zeta_i u_i.$

Note that the $\mathbf{L}^p$ norms define the 
same topologies for velocities $\mathbf{u}$ of the form 
$\sum_{i=1}^m \zeta_i u_i.$ 
Suppose $\mathbf{u}, v,w$ and $\tilde{\mathbf{u}}, \tilde v, \tilde w$ are two such solutions.
Define 
\begin{equation*}
\bar{\mathbf{u}} = \mathbf{u} - \tilde{\mathbf{u}}, \quad 
\bar{v} = v - \tilde{v}, \quad
\bar{w} = w - \tilde{w}, \quad
\bar{\phi} = \phi - \tilde{\phi}. 
\end{equation*}
Note that $\bar v$ is a solution to the problem
\begin{gather*}
\bar v_t + \mathbf{u} \cdot \nabla \bar v + \bar{\mathbf{u}}\cdot \nabla \tilde v=
\nabla \cdot (\nabla \bar v - v \nabla \bar \phi  - \bar v \nabla \tilde \phi)\\
 \frac{\partial \bar v}{\partial \nu} - v \frac{\partial \bar \phi}{\partial \nu} - \bar v \frac{\partial \tilde \phi}{\partial \nu}= 0,\quad
\mbox{ on } \partial \Omega \times (0,\infty), \quad 
\bar v(x,0) = 0,\quad   x \in \Omega.
\end{gather*}
Define $\Phi(t) = \|\bar{\mathbf{u}}, \bar v, \bar w\|_{\mathbf{L}^2}^2.$ 
Multiplying the first equation by $\bar v$ and integrating by parts we find
\begin{equation}
\label{gogo}
\frac{d}{dt}\frac{1}{2}\|\bar v\|_{\mathbf{L}^2}^2 + \|\nabla \bar v\|_{\mathbf{L}^2}^2
= (\Theta(t), \nabla \bar v).
\end{equation}
for $\Theta(t)$ in terms of the two solutions. 
In each of the bilinear terms of $\Theta,$ there will be one cross term
involving the difference of the solutions. Then
\begin{equation*}
\Theta(t) \leq C\Phi(t)
\end{equation*}
for some constant $C$ depending only on $\Omega, m$ and the initial data.
We omit the details as they follow in a straightforward way using the 
equivalence of the $\mathbf{L}^p$ norms on the velocities,
 \eqref{firstLp} and the regularity property of the Poisson equation.
An analogous estimate applies to $w.$ 

It is also clear from \eqref{ode}, that there is $C$ depending only on $\Omega, m$ and the initial data for which 
\begin{equation}
\label{velvel}
\frac{d}{dt}  \|\bar{\mathbf{u}}\|_{\mathbf{L}^2}^2 \leq C\|\bar{\mathbf{u}}\|_{\mathbf{L}^2}^2 + \|\Xi\|_{\mathbf{L}^1}^2 
\end{equation}
where $\Xi = \Delta \phi \nabla \bar \phi - \Delta \bar \phi \nabla \tilde \phi,$  and another such $C$ such that 
\begin{equation*}
 \|\Xi\|_{\mathbf{L}^1}^2 \leq C\Phi(t).
\end{equation*}

Adding \eqref{gogo} to the corresponding inequality for $w$ and to \eqref{velvel},
we have shown that there is 
$C$ depending only on $\Omega,$ the initial data and $m$ for which
\begin{equation*}
\frac{d}{dt} \Phi(t) \leq C \Phi(t) 
\end{equation*}
Since $\Phi(0) = 0,$ we infer
\begin{equation*}
\sup_{t \in (0,T_0)} \Phi(t) \leq C T_0 \sup_{t \in (0,T_0)} \Phi(t).
\end{equation*}
This inequality implies $\sup_{t \in (0,T_0)} \Phi(t)= 0 $  provided $T_0 \leq NC^{-1}.$
The uniqueness is now established.

The $\mathbf{C}^{2+\alpha}(\mathbf{Q}_{T_0})$ regularity of $v_m$ and $w_m$
now follows from proposition \ref{generalvw} and that fact that 
$\mathbf{u_m} \in \mathbf{Lip}(\mathbf{Q}_{T_0}).$  The regularity of 
$\mathbf{u}_m$
follows from the easy observation that the right hand side of \eqref{ode}
is H\"older continuously differentiable in $t$ 
and that $D_x^l \zeta_i \in \mathbf{C}^{\alpha}(\mathbf{Q}_{T_0})$ for all
$l \geq 0.$
\end{proof}

\begin{lemma}
\label{L6lemma}
Let $\Omega$ be an open subset of $\mathbf{R}^n, n =2,3.$ 
If $u \in \mathbf{H}^2(\Omega) \cap \mathbf{H}^1_0(\Omega),$
then 
\begin{gather*}
\|\nabla u\|_{\mathbf{L}^6} \leq 
6^{\frac{1}{2}}\|\nabla u\|_{\mathbf{L}^2}^{\frac{1}{2}}
\|\nabla^{2} u\|_{\mathbf{L}^3}^{\frac{1}{2}},\quad \mathrm{dim} \,\Omega = 2,\\
\|\nabla u\|_{\mathbf{L}^6}\leq 6^{\frac{2}{3}}
\|\nabla u\|_{\mathbf{L}^2}^{\frac{1}{3}}
\|\nabla^2 u\|_{\mathbf{L}^3}^{\frac{2}{3}},\quad \mathrm{dim} \,\Omega = 3.
\end{gather*}
\end{lemma}
\begin{proof}
We prove the $2$-dimensional case first.
Without loss of generality we may assume $u \in C^{\infty}_0(\mathbf{R}^2).$
Let $v= u_x.$ Then $v$ has compact support.
Clearly 
\begin{equation*}
v^3(x,y) = 3\int_{-\infty}^y v^2(x,s)v_{y}(x,s)\,ds.  
\end{equation*}
Thus, 
\begin{equation*}
|v|^3(x,y) \leq 3 \int_{\mathbf{R}} v^2(x,s)|v_{y}(x,s)|\,ds \equiv f(x).
\end{equation*}
Similarly, 
\begin{equation*}
|v|^3(x,y) \leq 3 \int_{\mathbf{R}} v^2(t,y)|v_{x}(t,y)|\,dt \equiv g(y).
\end{equation*}
Then,
\begin{equation*}
\iint_{\mathbf{R}^2} v^6(x,y)\,dx \,dy \leq \int_{\mathbf{R}} f(x) \,dx \int_{\mathbf{R}} g(y) \,dy.
\end{equation*}
However, by H\"older's inequality and the relation $u = v_x,$
\begin{equation*}
\int_{\mathbf{R}} f(x)\,dx \leq 3\iint_{\mathbf{R}^2}|v|^2 |\nabla v| \,dx \,ds   
\leq 3\|u_x\|_{\mathbf{L}^3}^2 \|\nabla^2 u\|_{\mathbf{L}^3}.
\end{equation*}
Bounding the integral of $g$ in the same way,
and using the interpolation $
\|\nabla u\|_{\mathbf{L}^3}^2 \leq 
\|\nabla u\|_{\mathbf{L}^2} \|\nabla u\|_{\mathbf{L}^6},$ we find
\begin{equation*}
\|u_x\|_{\mathbf{L}^6}^6 \leq 
9\|u_x\|_{\mathbf{L}^3}^4 \|\nabla^2 u\|_{\mathbf{L}^3}^2 \leq 9 
\|u_x\|_{\mathbf{L}^2}^2\|u_x\|_{\mathbf{L}^6}^2\|\nabla^2 u\|_{\mathbf{L}^3}^2
\end{equation*}
or
\begin{equation*}
\|u_x\|_{\mathbf{L}^6}^2 \leq 3\|\nabla u\|_{\mathbf{L}^2}\|\nabla^2 u\|_{\mathbf{L}^3}.
\end{equation*}
Arguing similarly with $u_y,$ we have, by the triangular inequality,
\begin{equation*}
\|\nabla u\|_{\mathbf{L}^6} 
\leq (\|u_x\|^2_{\mathbf{L}^6} + \|u_y\|_{\mathbf{L}^6}^2)^{\frac{1}{2}}
\leq \sqrt{6}\|\nabla u\|_{\mathbf{L}^2}^{\frac{1}{2}}\|\nabla^2 u\|_{\mathbf{L}^3}^{\frac{1}{2}}.
\end{equation*}
This proves the first part of the lemma.

Now we consider the $3$-dimensional case.  
For $z \in \mathbf{R},$ 
\begin{equation*}
\begin{split}
&\|\nabla u(\cdot,\cdot, z)\|_{\mathbf{L}^3(\mathbf{R}^2)}^3
= \iint_{\mathbf{R}^2} |\nabla u|^3(x,y,z) \,dx \,dy \\
&= 3\iint_{\mathbf{R}^2} \int_{-\infty}^z (|\nabla u| \nabla u \cdot \nabla u_z)(x,y,s) \,ds \,dx \,dy \\
&\leq 3\iiint_{\mathbf{R}^3} |\nabla u|^2|\nabla^2 u| \,dx \,dy \,dz 
\leq  3\|\nabla u\|_{\mathbf{L}^3}^{2}\|\nabla^2 u\|_{\mathbf{L}^3}.
\end{split}
\end{equation*}
By what was shown in the $2$-dimensional case,
\begin{equation*}
\begin{aligned}
&\iiint_{\mathbf{R}^3} |\nabla u|^6  \,dx \,dy \,dz 
\leq \\
&72\int_{\mathbf{R}} \|\nabla u(\cdot,\cdot, z)\|_{\mathbf{L}^3(\mathbf{R}^2)}^{4}
\ |\nabla^2 u(\cdot,\cdot, z)\|_{\mathbf{L}^3(\mathbf{R}^2)}^{2} \,dz.
\end{aligned}
\end{equation*}
Consequently,
\begin{equation*}
\begin{split}
\|\nabla u\|_{\mathbf{L}^6}^6 
&\leq 
72  \cdot 3 \|\nabla u\|_{\mathbf{L}^3}^{2}\|\nabla^2 u\|_{\mathbf{L}^3} \times\\
&\int_{\mathbf{R}}\|\nabla u(\cdot,\cdot, z)\|_{\mathbf{L}^3(\mathbf{R}^2)}
\|\nabla^2 u(\cdot,\cdot, z)\|_{\mathbf{L}^3(\mathbf{R}^2)}^{2} \,dz\\
&\leq 6^{3}  \|\nabla u\|_{\mathbf{L}^3}^{3}\|\nabla^2 u\|_{\mathbf{L}^3}^3.
\end{split}
\end{equation*}
Interpolating $\mathbf{L}^3 \subset \mathbf{L}^2 \cap \mathbf{L}^6$ again,
\begin{equation*}
\|\nabla u\|_{\mathbf{L}^6}^6  \leq 6^3  \|\nabla u\|_{\mathbf{L}^2}^{\frac{3}{2}}
 \|\nabla u\|_{\mathbf{L}^6}^{\frac{3}{2}}\|\nabla^2 u\|_{\mathbf{L}^3}^3.
\end{equation*}
Dividing and taking the appropriate power now gives the second part of the lemma.
\end{proof}

\begin{lemma}
\label{loglemma}
Let $\Omega$ be a smooth bounded subset of $\mathbf{R}^n$ 
and $p =\frac{n}{n-1}$ and $q = \frac{3}{2}p.$
If there is a constant $F$ so that 
\begin{equation*}
\|f\|_{\mathbf{L}\log \mathbf{L}} \leq F
\end{equation*}
 and $\epsilon > 0,$ 
then there is $\clabel{logcon} = \cref{logcon}(\Omega, n, \epsilon, F)$ for which
\begin{equation*}
\|f\|_{\mathbf{L}^{q}}^{q} \leq \epsilon \|\nabla f\|_{\mathbf{L}^2}^{p} + (1 + \|f\|_{\mathbf{L}^1}^{p})\cref{logcon}.
\end{equation*}  
\end{lemma}
\begin{proof}
Without loss of generality, assume $f$ 
is nonnegative.
Let $E : \mathbf{H}^1(\Omega) \rightarrow \mathbf{H}^1(\mathbf{R}^n)$
be a continuous extension with $\mathrm{Supp}\,Ef \subset \mathbf{B}(R)$ and  
$\|Ef\|_{\mathbf{L}\log \mathbf{L}} \leq \tilde F$
 where $\tilde F = \tilde F(\Omega, n, F) < \infty.$
Denote $\tilde f = Ef.$  For $M > 0,$
\begin{equation*}
\int_{\Omega} f^{q} \,dx
\leq \int_{\mathbf{R}^n} \tilde f ^{q} \,dx 
\leq |\mathbf{B}(R)|M^{q} + \int_{\mathbf{R}^n} (\tilde f - M)_+^{q} \,dx  
\end{equation*}
By the Gagliardo-Nirenberg-Sobolev and H\"older inequalities we have
\begin{equation*}
\begin{split}
\int_{\mathbf{R}^n} (\tilde f - M)_+^{q} \,dx
&\leq C(n)\left(\int_{\{\tilde f > M\}} |D\tilde f|\tilde f^{1/2} \,dx\right)^{p}\\
&\leq C(n)\|D\tilde f\|_{\mathbf{L}^{2}(\mathbf{R}^n)}^{p}
\left(\int_{\{\tilde f > M\}} \tilde f \,dx\right)^{\frac{p}{2}}.
\end{split}
\end{equation*}
Note that 
\begin{equation*}
\int_{\{\tilde f > M\}} \tilde f \,dx
\leq \frac{1}{\log M} \int_{\{\tilde f > M\}} \tilde f  \log \tilde f \,dx
\leq \frac{\tilde F}{\log M}.
\end{equation*}
Combining the above and applying the continuity of $E,$ 
\begin{equation*}
\int_{\mathbf{R}^n} (\tilde f - M)_+^{q} \,dx
\leq C(\Omega, n)\left(\frac{\tilde F}{\log M}\right)^{\frac{p}{2}}\|f\|_{\mathbf{H}^1}^{p}. 
\end{equation*}
Cetainly, by the Poincare inequality,
$\|f\|_{\mathbf{H}^1} \leq C(\Omega) (\|Df\|_{\mathbf{L}^2} + \|f\|_{\mathbf{L}^1}).$
Combining this with the results from above
\begin{equation*}
\|f\|_{\mathbf{L}^{q}}^{q} \leq C(\Omega, n)\left(\frac{\tilde F}{\log M}\right)^{\frac{p}{2}}
\|Df\|_{\mathbf{L}^2}^{p} + C(\Omega, n, M)(1 + \|f\|_{\mathbf{L}^1}^{p}).
\end{equation*}
Choosing $M$ sufficiently large, the lemma follows.
\end{proof}

\begin{lemma}[\cite{GILBARG}, corollary 9.10.]
\label{poissonLp}
If $\Omega$ is a smooth, bounded, open subset of $\mathbf{R}^n,$
$u \in \mathbf{W}^{2,p} \cap \mathbf{H}_0^1$ and $1 < p < \infty,$ then 
\begin{equation*}
\|\nabla^2 u\|_{\mathbf{L}^p} \leq C(\Omega,p ) \|\Delta u\|_{\mathbf{L}^p}.
\end{equation*} 
\end{lemma}

\begin{lemma}
\label{trilinear}
If $\mathrm{dim}\,\Omega = 2$ and 
there are constants $a_1, a_2, a_3$ so that 
\begin{equation*}
\|\nabla \phi\|_{\mathbf{L}^2}^2 \leq a_1, \quad
\|v, w\|_{\mathbf{L}\log \mathbf{L}} \leq a_2, \quad 
\|v,w\|_{\mathbf{L}^1} \leq a_3,
\end{equation*}
$\phi \in \mathbf{H}^2 \cap \mathbf{H}^1_0$ solves $\Delta \phi = v - w$ 
and $\epsilon, \delta > 0,$ then there is $\clabel{trilem} = \cref{trilem}(\Omega, \epsilon,\delta, a_1, a_2, a_3)$ with
\begin{equation*}
\left|\int_{\Omega} \nabla \phi \cdot (v \nabla v - w \nabla w)\,dx \right|
\leq \epsilon \|\nabla v, \nabla w\|_{\mathbf{L}^2}^2 - \delta  \|v,   w\|_{\mathbf{L}^2}^2 + \cref{trilem}.
\end{equation*}
\end{lemma}
\begin{proof}
By lemma \ref{L6lemma}, with $\mathrm{dim}\,\Omega = 2,$ 
 and lemma \ref{poissonLp}, there is $C_1 = C_1(\Omega, a_1)$ for which
\begin{equation*}
\|\nabla \phi\|_{\mathbf{L}^6} \leq 6^{\frac{1}{2}}\|\nabla \phi\|_{\mathbf{L}^2}^{\frac{1}{2}}
\|\nabla^2 \phi\|_{\mathbf{L}^3}^{\frac{1}{2}}
\leq C_1\|v, w\|^{\frac{1}{2}}_{\mathbf{L}^3}.
\end{equation*}
By lemma \ref{loglemma}, with $\mathrm{dim}\,\Omega = n = 2,$
there is $C_2 = C_2(a_2, a_3, \cref{logcon})$ so that 
\begin{equation*}
\|v, w\|_{\mathbf{L}^3}^{\frac{3}{2}} \leq
\epsilon_1\|\nabla v,\nabla w\|_{\mathbf{L}^2} + C_2.
\end{equation*}
If $\epsilon_2 > 0,$ then by Nash's inequality, there is $C_3 = C_3(\Omega, \epsilon_2, a_3)$ so that
\begin{equation*}
\|v,w\|_{\mathbf{L}^2}^2 \leq \epsilon_2 \|\nabla v, \nabla w\|_{\mathbf{L}^2}^2  + C_3.
\end{equation*}
Combining these, we estimate the trilinear term.
By H\"older's inequality, 
\begin{equation*}
\begin{split}
\left|\int_{\Omega} \nabla \phi \cdot (v \nabla v - w \nabla w)\,dx \right| 
&\leq |\nabla \phi|_6 \|v,w\|_{\mathbf{L}^3} \|\nabla v, \nabla w\|_{\mathbf{L}^2}\\
&\leq C_1\|v, w\|_{\mathbf{L}^3}^{\frac{3}{2}} \|\nabla v, \nabla w\|_{\mathbf{L}^2}\\
&\leq C_1(\epsilon_1 \|\nabla v, \nabla w\|_{\mathbf{L}^2} + C_2)\|\nabla v, \nabla w\|_{\mathbf{L}^2}\\
&\leq (C_1 \epsilon_1 + \epsilon_3)\|\nabla v, \nabla w\|_{\mathbf{L}^2}^2
+ \epsilon_3^{-1}(2C_1C_2)^2.
\end{split}
\end{equation*}
Choose $\epsilon_1$ and $\epsilon_3$ so that $C_1 \epsilon_1 + \epsilon_3 =  \epsilon/2.$
Also, choose $\epsilon_2 = \delta^{-1}\epsilon/2.$  Then
\begin{equation*}
\begin{split}
\frac{\epsilon}{2}\|\nabla v, \nabla w\|_{\mathbf{L}^2}^2
&\leq \epsilon \|\nabla v, \nabla w\|_{\mathbf{L}^2}^2 - \epsilon_2 \delta \|\nabla v, \nabla w\|_{\mathbf{L}^2}^2\\ 
&\leq \epsilon \|\nabla v, \nabla w\|_{\mathbf{L}^2}^2 - \delta \|v,  w\|_{\mathbf{L}^2}^2 + \delta C_3.
\end{split}
\end{equation*} 
The lemma follows with $\cref{trilem} = \delta C_3 + \epsilon_3^{-1}(2C_1C_2)^2.$
\end{proof}

\begin{lemma}
\label{extends}
Let $\mathrm{dim} \, \Omega = 2$ 
and let $\mathbf{u}_m, v_m, w_m$ be the solution obtained from lemma \ref{smalltime}.
For $r > 0,$ let $\psi(r) = r \log r - r + 1$ and define
\begin{equation*}
{W} =
\int_{\Omega} \psi(v_m) + \psi(w_m) \,dx   + \frac{1}{2}\|\nabla \phi_m\|_{\mathbf{L}^2}^2 + \frac{1}{2}\|{\bf u}_m\|_{\mathbf{L}^2}^2.
\end{equation*}
then there is $\clabel{extlem} = \cref{extlem}(\Omega, W(0), \|v_0\|_{\mathbf{L}^1}, \|w_0\|_{\mathbf{L}^1})$ such that if 
\begin{equation*}
\|v_0, w_0\|_{\mathbf{L}^2}^2 \leq \cref{extlem},
\end{equation*}
then
\begin{equation*}
\sup_{t \in (0,T_0)} \|v_m, w_m\|_{\mathbf{L}^2}^2 \leq \cref{extlem}.
\end{equation*}
\end{lemma}
\begin{proof}
From lemma \ref{generalvw}, 
$v_m$ and $w_m$ are positive on $\overline{\mathbf{Q}_{T_0}}.$
Mimicking the calculation of the basic energy law \eqref{apriori}, 
differentiation of $W$ 
with respect to $t$ gives
\begin{equation}
\label{apriori2}
\begin{split}
\frac{d{W}}{dt} 
&=- \int_{\Omega} v_m |\nabla \log (v_m e^{-\phi_m})|^2 + w_m |\nabla \log (w_m e^{\phi_m})|^2 +  |\nabla {\bf u}_m|^2 \,dx\\
& \leq0.
\end{split}
\end{equation}
for  all $0 \leq t \leq T_0.$ Here we used the identity
\begin{equation*}
(({\bf u}_m)_t, {\bf u}_m) =  -\|\nabla {\bf u}_m\|^2_{\mathbf{L}^2}  + ({\bf u}_m\cdot \nabla \phi_m, \Delta \phi_m).
\end{equation*}
Recall that the intregral of $v_m$ and $w_m$ are conserved quantities. 
\begin{equation}
\label{integral}
\|v_m\|_{\mathbf{L}^1} = \|v_0\|_{\mathbf{L}^1},\quad \|w_m\|_{\mathbf{L}^1} =\|w_0\|_{\mathbf{L}^1}.
\end{equation}
Then
\eqref{apriori2} implies 
\begin{equation}
\label{LogL2uniform}
\|v_m, w_m\|_{\mathbf{L} \log \mathbf{L}}
+ \|\nabla \phi_m\|_{\mathbf{L}^2}^2 
\leq W(0).
\end{equation}
Define $\omega(t) =  \|v_m,w_m\|_{\mathbf{L}^2}^2$ and $\zeta(t) = \|\nabla v_m, \nabla w_m\|_{\mathbf{L}^2}^2.$
Let $\epsilon = \delta =\frac{1}{2}.$  If  
$\cref{extlem} = 2\cref{trilem},$ then  
by lemma \eqref{pest} with $p=2,$ 
\begin{equation}
\label{coolstuff}
\frac{d}{dt} \omega + \omega + \zeta \leq \cref{extlem}
\end{equation}
This inequality implies  $\omega(t) \in [0,\max\{\cref{extlem}, \omega(0)\}]$ for all $t\in [0,T_0].$
\end{proof}

\begin{proof}[Proof of theorem \ref{mglobal}]
Define 
\begin{equation*}
\clabel{cglob} = \max\{\cref{extlem}, W(0), \|\mathbf{u}_0, v_0, w_0\|_{\mathbf{L}^2}^2\}.
\end{equation*}
The energy decay \eqref{apriori2} implies that
 $\|\mathbf{u}_m\|_{\mathbf{L}^2}^2(t) \leq W(t) \leq W(0) \leq \cref{cglob}$ for all $t \in [0,T_0].$
The conclusion of lemma \ref{extends} 
implies that  $\|v_m, w_m\|_{\mathbf{L}^2}^2(t) \leq \cref{cglob}$  for all $t \in [0,T_0].$
The solution obtained from lemma \ref{smalltime}
enjoys 
\begin{equation}
\label{globalL2bound}
\|\mathbf{u}_m, v_m, w_m\|_{\mathbf{L}^2}^2 \leq 2\cref{cglob}, 
\quad \forall t \in [0, T_0].
\end{equation}
Returning to lemma \ref{smalltime}, the solution may be extended 
to an interval $[0, T_0 + \delta]$ where $\delta$ depends only on $\cref{cglob}.$
The conclusion of theorem \ref{mglobal} now follows from repeated application of this extension 
property;
the regularity of $\mathbf{u}_m, v_m$ and $w_m$ follows as in lemma \ref{smalltime}
because \eqref{smoothness} is satisfied by $\tilde{\mathbf{u}}_0 = u_m(T_0),  \tilde v_0 = v_m(T_0),  \tilde w_0 = w_m(T_0).$ \end{proof}

\begin{lemma}  
Let $\mathrm{dim}\,\Omega = 2,$  $T \in (0,\infty)$ and $\mathbf{u}_m, v_m, w_m$  be
the solution obtained from theorem \ref{mglobal}.  Then there is $\clabel{clem10} < \infty$ 
depending only on the $\mathbf{L}^2$ norms of $\mathbf{u}_0, v_0, w_0$ 
and $W(0)$ but 
independent of $m$ so that
\begin{equation}
\label{uvwbound}
\begin{aligned}
\sup_{t \in (0,T)} \|\mathbf{u}_m, v_m, w_m\|_{\mathbf{L}^2}
+ \int_0^T  \|\nabla \mathbf{u}_{{m}}, \nabla v_{{m}}, \nabla w_{{m}}\|_{\mathbf{L}^2}^2 \,ds \\
+ \int_0^T \left\|\frac{d\mathbf{u}_{{m}}}{dt}\right\|_{\mathbf{V}^*}^2 
+  \left\|\frac{dv_{{m}}}{dt}, \frac{dw_{{m}}}{dt}\right\|_{\mathbf{H}^{-1}}^2\,ds 
\leq \cref{clem10}.
\end{aligned}
\end{equation}
\end{lemma}
\begin{proof}
One can construct a bound of the first two terms in \eqref{uvwbound}
by simply integrating  \eqref{apriori2}, \eqref{globalL2bound} and \eqref{coolstuff}.
In particular, the $\mathbf{L}^2$ norms of $v,w$ are bounded in $t$
uniformly in $m.$ 
From \eqref{specialconsequence}, with $p = 2, k = 0,$
$\|\mathbf{f}_m\|_{\mathbf{V}^*}$ is then also bounded in $t$ uniformly in $m.$
It is then standard (c.f. \cite{TEMAM01}, theorem 3.1) to check that 
\begin{equation*}
\frac{d\mathbf{u}_{{m}}}{dt} \in \mathbf{L}^2(0,T; \mathbf{V}^*)
\end{equation*}  with norm independent of $m.$

It is straightforward to check, using lemma \ref{ladyzhenskaya} and the 
Sobolev embedding $\mathbf{H}^1 \subset \mathbf{L}^{6}\subset \mathbf{L}^{3} ,$ that 
\begin{equation*}
\begin{split}
&\left\| \frac{dv_{{m}}}{dt}\right\|_{\mathbf{H}^{-1}}
= \sup_{\nu \in \mathbf{H}^1, \|\nu\|_1 = 1}(\nabla v_m - v_m \nabla \phi_m - v_m \mathbf{u}_m, \nabla \nu)\\
&\leq (\|v_m\|_{\mathbf{H}^1} + \|\mathbf{u}_m\|_{\mathbf{L}^4}\|v_m\|_{\mathbf{L}^4} + \|v_m\|_{\mathbf{L}^3}\|\nabla \phi\|_{\mathbf{L}^6}).
\end{split}
\end{equation*}
Using the Sobolev inequality, we can easily bound the right hand side by
\begin{equation*}
C(\|v_m\|_{\mathbf{H}^1} + \|v_m\|_{\mathbf{H}^1})
\end{equation*}
for some constant depending only on $\Omega$ and $\cref{cglob}.$ 
Applying similar reasoning to $w_m,$ we see then that 
\begin{equation*}
\frac{dv_{{m}}}{dt}, \frac{dw_{{m}}}{dt}\in \mathbf{L}^2(0,T; \mathbf{H}^{-1})
\end{equation*}  with norm independent of $m.$
\end{proof}

\begin{lemma}  
\label{f'lemma}
Let $\mathrm{dim}\,\Omega = 2,$  $T \in (0,\infty)$ and $\mathbf{u}_m, v_m, w_m$  be
the solution obtained from theorem \ref{mglobal}.  Let
\begin{equation*}
\mathbf{f}_m = \Delta \phi_m \nabla \phi_m.
\end{equation*}
Then there is $\clabel{clem11} < \infty$ 
depending only on the $\mathbf{L}^2$ norms of $\mathbf{u}_0, v_0, w_0$ 
and $W(0)$ but 
independent of $m$ so that
\begin{equation}
\label{fpbound}
\sup_{t \in (0,T)} \| \mathbf{f}_m \|_{\mathbf{V}^*}^2
+ \int_0^T \left \|\frac{d\mathbf{f}_m}{d t}\right\|_{\mathbf{V}^*}^2 \,ds \leq \cref{clem11}.\\
\end{equation}
\end{lemma}
\begin{proof}
\eqref{fpbound} is an easy consequence of  \eqref{specialconsequence} and \eqref{uvwbound}.
\end{proof}

\subsection{Proof of Theorem \ref{globalex}}
Let $\mathbf{u_0} \in \mathbf{H}, v_0, w_0 \in \mathbf{L}^2.$ 
Let $\{v^{\epsilon}_0, w^{\epsilon}_0\}$ 
be a sequence of functions satisfying \eqref{smoothness}
with $v^{\epsilon}_0, w^{\epsilon}_0 \rightarrow v_0, w_0$ in $\mathbf{L}^2$ as 
$\epsilon \downarrow 0.$  Let $T \in (0,\infty)$ and $m$ be a positive integer.
Let $u_m^{\epsilon}, v_m^{\epsilon}, w_m^{\epsilon}$ be the solution obtained from theorem \ref{mglobal}
with initial data $\mathbf{u}_0, v^{\epsilon}_0, w^{\epsilon}_0.$
The bound \eqref{uvwbound} is independent of $m$ and $\epsilon.$ 
Applying Lion's compactness theorem (\cite{TEMAM01}, theorem 2.3),
for some subsequence of $\{\mathbf{u}_m^{\epsilon}, v_m^{\epsilon},w_m^{\epsilon}\}_{m \uparrow \infty, \epsilon \downarrow 0},$ 
 \begin{equation*}
\begin{split}
\mathbf{u}_m^{\epsilon} \rightarrow \mathbf{u} &\mbox{ weakly in } \mathbf{L}^2((0,T); \mathbf{V}), \mbox{ weak-* in } \mathbf{L}^{\infty}((0,T); \mathbf{H}),\\
&\mbox{ and strongly in } \mathbf{L}^2((0,T); \mathbf{H}),\\
v_m^{\epsilon},w_m^{\epsilon}  \rightarrow v,w &\mbox{ weakly in } \mathbf{L}^2((0,T); \mathbf{H}^1), \mbox{ weak-* in } \mathbf{L}^{\infty}((0,T); \mathbf{L}^2),\\
&\mbox{ and strongly in } \mathbf{L}^2((0,T); \mathbf{L}^2).\\
\end{split}
\end{equation*}
It is straightforward to check (e.g. \cite{TEMAM01}, section 3.2) that additionaly
\begin{gather*}
\mathbf{u_m^{\epsilon}} \cdot \nabla \mathbf{u}_m^{\epsilon} \rightarrow \mathbf{u}\cdot \nabla \mathbf{u} \mbox{ in } \mathbf{L}^1(\mathbf{Q}_T)\\
\mathbf{u_m^{\epsilon}} \cdot \nabla v_m^{\epsilon}, \mathbf{u_m^{\epsilon}} \cdot \nabla w_m^{\epsilon}\rightarrow \mathbf{u}\cdot \nabla v, \mathbf{u}\cdot \nabla w \mbox{ in } \mathbf{L}^1(\mathbf{Q}_T),\\
(v_m^{\epsilon} \nabla \phi_m^{\epsilon}, \nabla \nu)  \rightarrow ( v \nabla \phi, \nabla \nu),  \quad (w_m^{\epsilon} \nabla \phi_m^{\epsilon}, \nabla \nu)  \rightarrow ( w \nabla \phi, \nabla \nu),\\
  \mbox{ in } \mathbf{L^1}(0,T) \quad  \forall \nu \in \mathbf{H}^1.
\end{gather*}
One easily verifies that $\langle \mathbf{u},v,w, \phi \rangle$ 
is a global weak solution of \eqref{NS}-\eqref{I}
(c.f. \cite{TEMAM01}, Chapter 3.)

Let $\mathbf{f} = \Delta \phi \nabla \phi$ and 
$\mathbf{f}_m^{\epsilon} = \Delta \phi_m^{\epsilon} \nabla \phi_m^{\epsilon}.$
By Fubini's theorem and the above, the functions $v_m^{\epsilon},w_m^{\epsilon}$ converge 
 to $v,w,$ in the strong topology of
 $\mathbf{L}^2(\mathbf{Q}_T).$ By regularity of solutions to the 
Poisson equation, $\nabla \phi_m^{\epsilon} \otimes \nabla \phi_m^{\epsilon}$ 
converges to $\nabla \phi \otimes \nabla \phi$ 
in the strong topology of
 $\mathbf{L}^2(\mathbf{Q}_T).$
Using the characterization \eqref{consequence}
and the estimate \eqref{fpbound}, one verifies through
multiplying by test functions that 
\begin{equation*}
\mathbf{f} \in \mathbf{L}^{\infty}((0,T); \mathbf{V}^*), \quad
\mathbf{f}_t \in \mathbf{L}^{2}((0,T); \mathbf{V}^*).
\end{equation*}
$\mathbf{u}$ is a weak solution of \eqref{genNS}.
Let $0 < s.$  
By proposition \eqref{uniformNS}  
we conclude 
$\mathbf{u}_t \in \mathbf{L}^2((s,T); \mathbf{H}^1) \cap 
\mathbf{L}^{\infty}((s,T); \mathbf{L}^2) .$
Applying the standard regularity results for the Navier-Stokes system to  \eqref{genNS},
we find $\mathbf{u} \in \mathbf{L}^{\infty}((s,T); \mathbf{H}^2).$
We conclude from the Sobolev embedding,
\begin{equation*}
\int_{\mathbf{Q}_{(s,T)}}
|\mathbf{u}_t|^{6}
+ |\nabla \mathbf{u}|^{6} \,dx \,dt   < 
\infty.
\end{equation*}
This is sufficient to conclude that 
\begin{equation*}
\mathbf{u} \in \mathbf{C}^{\alpha}(\mathbf{Q}_{s,T})
\end{equation*}
for some $\alpha > 0.$
By choosing $p = 2,$  
proposition \ref{generalvw} implies that
$v,w \in \mathbf{L}^{\infty}(\mathbf{Q}_{s,T}).$
The $\mathbf{C}^{2 + \alpha}$ 
regularity  $v$ and $w$ follows as in proposition \ref{generalvw},
although here the boundary data on the bottom of $\mathbf{Q}_T$ 
may not be smooth.  By \cite{LIEBERMAN}, Theorem 6.44 
$v, w \in \mathbf{C}^{\alpha}(\mathbf{Q}_{s,T}).$ 
Then $\nabla \phi \in \mathbf{C}^{\alpha}(\mathbf{Q}_{s,T})$
as well and from the material in  \cite{LIEBERMAN}, Chapter 5,
$v, w \in \mathbf{C}^{2 +\alpha}(\mathbf{Q}_{s,T}).$
Serrin's result (c.f. \cite{Se62}), which states that 
a weak solution of the Navier-Stokes equations in two dimensions 
with $\mathbf{u}_t \in \mathbf{L}^{\infty}((s,T); \mathbf{L}^{2})$
and $\mathbf{f} \in \mathbf{C}^{\alpha}$ (which is indeed the case here),
enjoys $\mathbf{u} \in \mathbf{C}^{2+\alpha}(\mathscr{R}).$

\section{Entropy and Decay}
\label{Dolbeaultsproof} 
In \cite{BiDo00}, Biler and Dolbeault studied
various convergence estimates for weak solutions of the
Debye-H\"uckel system to the static solution.
The trick was to observe that in certain instances 
entropy production terms can be bounded in terms of a \emph{relative entropy}.
The same observations can be made in the case of the
system modeling electro-hydrodynamics.  
Theorem \ref{l1ass}, which is modeled  after
the work of Biler and Dolbeault,  
follows by proving that the relative entropy of the
electro-hydrodynamic system decays atleast  
with a rate depending only on ${\Omega}.$

Let $v,w \in \mathbf{L}^2$ and $\phi \in \mathbf{H}^1_0$
be a solution of the Poisson equation \eqref{poisson}.
Let $\langle v_{\infty}, w_{\infty}, \phi_{\infty}\rangle$
be the stationary solution from Corollary \ref{statexist}.
Define the \emph{relative entropy} 
\begin{equation*}
W_{\mathrm{rel}}
\equiv \int_{\Omega} \psi_{v_{\infty}}(v) + \psi_{w_{\infty}}(w) \,dx
+ \frac{1}{2}\|\nabla \phi - \nabla \phi_{\infty}\|^2_{\mathbf{L}^2} + \frac{1}{2}\|{\bf u}\|^2_{\mathbf{L}^2}.
\end{equation*}
Assume $\Omega$ is bounded and uniformly convex
subset of $\mathbf{R}^n.$
We will show 
that there is a constant $\lambda_1$ depending only on
$\Omega$ so that that if $\langle \mathbf{u}, v,w,\phi\rangle$ is 
global weak solution of \eqref{NS}-\eqref{I},
then 
\begin{equation}
\label{Wreldecay}
\frac{d W_{\mathrm{rel}}}{dt}
\leq -\lambda W_{\mathrm{rel}}, \quad \forall t \in (0,\infty).
\end{equation}
\begin{lemma}[Generalized Csiszar-Kullback Inequality,
\cite{UnArMaTo00}]
Let $v,w \in \mathbf{L}\log\mathbf{L}$ and $\phi \in \mathbf{H}^1_0.$
Then
\begin{equation}
\label{csiszar}
\|v- v_{\infty}\|_{\mathbf{L}^1} 
+ \|w - w_{\infty}\|_{\mathbf{L}^1}
+ \|\nabla \phi - \nabla \phi_{\infty}\|^2_{\mathbf{L}^2} 
+ \|u\|^2_{\mathbf{L}^2} \leq 
4W_{\mathrm{rel}}.
\end{equation}
\end{lemma}
Combining \eqref{Wreldecay} with \eqref{csiszar}
and setting $e_1 = 4W_{\mathrm{rel}}(0),$ theorem \eqref{l1ass} is now proved.

We now prove \eqref{Wreldecay}, we need
the following lemmas. 
\begin{lemma} 
Let $v,w \in \mathbf{L}^2$ and $\phi \in \mathbf{H}^1_0$
be a solution of the Poisson equation \eqref{poisson}.
Let 
\begin{equation*}
W_{\infty}
\equiv \int_{\Omega} \psi(v_{\infty}) + \psi(w_{\infty})\,\mathrm{d}x
+ \frac{1}{2}\|\nabla \phi_{\infty}\|_{\mathbf{L}^2}^2
\end{equation*}
and 
\begin{equation*}
{v_{M}} \equiv  \int_{\Omega} v_0 \, \mathrm{d}x  \frac{e^{\phi}}{\int_{\Omega} e^{\phi} \,dy},
\quad  {w_{M}} \equiv  \int_{\Omega} w_0 \,\mathrm{d}x  \frac{e^{-\phi}}{\int_{\Omega} e^{-\phi}\,dy }.
\end{equation*}
Let $J$ be defined by \eqref{CONVEX} and 
$W$ defined by \eqref{ENTROPY}.
Then 
 \begin{align}
\label{wwrel}
W_{\mathrm{rel}} &=  W + W_{\infty}\\
\label{wwmax}
& = \int_{\Omega} \psi_{v_M}(v) + \psi_{w_M}(w) + \frac{1}{2}|u|^2 \,dx
+ J[\phi_{\infty}] - J[\phi].
\end{align}
\end{lemma}

\begin{proof}
Equations \eqref{wwrel} and \eqref{wwmax} follow from elementary manipulations
involving the definition of the Maxwellians $v_M$ and $w_M$ and relations \eqref{poisson}, \eqref{D} and 
\begin{equation*}
\int_{\Omega} v \,dx = \int _{\Omega} v_{\infty} \,dx = \int _{\Omega} v_{0} \,dx,
\quad \int_{\Omega} w \,dx = \int _{\Omega} w_{\infty} \,dx= \int _{\Omega} w_{0} \,dx.
\end{equation*}
\end{proof}

\begin{lemma}[\cite{BiDo00}]
\label{BiDolemma}
Let $\Omega$ be a bounded, uniformly convex subset of $\mathbf{R}^n.$ 
There exists $\clabel{lambda1}$ depeding only on $\Omega$ 
so that for all $v,w \in \mathbf{L}\log\mathbf{L}, \phi \in \mathbf{H}^1_0,$
\begin{equation*}
\begin{aligned}
\int_{\Omega} \psi_{v_{M}}(v) + \psi_{w_{M}}(w)& \,\mathrm{d}x \\
&\leq \cref{lambda1} \int_{\Omega} v |\nabla\log (ve^{-\phi})|^2 + w |\nabla \log(w e^{\phi})|^2\,\mathrm{d}x.
\end{aligned}
\end{equation*}
\end{lemma}
\begin{proof}[Proof of \eqref{Wreldecay}]
Note that $W_{\infty}$ is independent of $t.$  Thus, if $\langle v,w,\phi, u \rangle$ 
is a global weak solution of \eqref{NS}-\eqref{I}, we have, according to \eqref{apriori}
and \eqref{wwrel},
\begin{equation*}
\label{wreldecay}
\begin{split}
&\frac{d W_{\mathrm{rel}}}{dt} =
 \frac{d W}{dt} \\ &= 
- \int_{\Omega} v |\nabla\log (ve^{-\phi})|^2 + w |\nabla \log(w e^{\phi})|^2 +  |\nabla u|^2 \,dx.
\end{split}
\end{equation*}
By the lemma \eqref{BiDolemma},
\begin{equation*}
\frac{d W_{\mathrm{rel}}}{dt} 
\leq - \cref{lambda1}\int_{\Omega} v\log\left( \frac{v}{v_M}\right) + w\log\left( \frac{w}{w_M}\right)
\,\mathrm{d}x  - \|\nabla u\|^2_{\mathbf{L}^2}.
\end{equation*}
By the Poincar\'e inequality, there is $C$ depending only on $\Omega$ for which 
\begin{equation*}
\frac{d W_{\mathrm{rel}}}{dt} 
\leq - \cref{lambda1}\int_{\Omega} v\log\left( \frac{v}{v_M}\right) + w\log\left( \frac{w}{w_M}\right)
\,\mathrm{d}x  - \frac{C}{2}\|u\|^2_{\mathbf{L}^2}.
\end{equation*}
Let $\lambda_1 = \min\{\cref{lambda1}, C\}.$ 
Then, applying \eqref{wwmax}, 
\begin{equation*}
\frac{d W_{\mathrm{rel}}}{dt} \leq -\lambda_1(W_{\mathrm{rel}} + J[\phi] - J[\phi_{\infty}]).
\end{equation*}
Since $\phi_{\infty}$ is a minimum of $J$ (see proof of theorem \ref{PBThm}), 
the difference of the last two terms
is positive.  This gives \eqref{Wreldecay}.
\end{proof}

\section{The 3 Dimensional Case: Small Data}
\label{globalbehave}

For domains with more general geometry, 
the techniques from the previous section are difficult to apply.
The main reason for the difficulty is that 
$W_{\mathrm{rel}}$ is not quadratic and it is not 
clear to what extent the logarthmic Sobolev inequality 
of lemma \eqref{BiDolemma} depend in the domain geometry.

To remedy this difficulty, we study a linearization 
of the relative entropy $W_{\mathrm{rel}}$ about the 
stationary solution $\langle v_{\infty}, w_{\infty}, \phi_{\infty} \rangle.$ 
We are able to show
that the linearization satisfies a decay estimate similar 
to \eqref{Wreldecay}.  This approach is succesful because
$W_{\mathrm{rel}}$ is locally quadratic about $\langle v_{\infty}, w_{\infty}, \phi_{\infty} \rangle.$ 

In order to construct the linearization
of the relative entropy $W_{\mathrm{rel}}$ about the stationary solution 
$\langle v_{\infty},w_{\infty}, \phi_{\infty} \rangle,$
consider an expansion of 
$\psi_r(s) = s \log (s/r) - s + r$  for $s \geq 0, r > 0.$
By Taylor's theorem,
\begin{equation*}
\begin{split}
\psi_r(s)  &= \psi_r(r) + (s - r) \psi'_r(r) + \frac{1}{2}(s-r)^2\psi''_r(r) + O((s-r)^3)\\
&=  \frac{1}{2}\frac{(s-r)^2}{r}  + O((s-r)^3).
\end{split}
\end{equation*}
If $s = v(x,t)$ (resp. $w(x,t)$) and $r = v_{\infty}(x)$ (resp. $w_{\infty}(x)$), the 
leading order term in $v, v_{\infty}, w, w_{\infty}$  of the integrand of $W_{\mathrm{rel}}$ is 
\begin{equation*}
\frac{1}{2}\frac{(v - v_{\infty})^2}{v_{\infty}} + \frac{1}{2}\frac{(w - w_{\infty})^2}{w_{\infty}}.
\end{equation*}
Motivated by this observation, we define
\begin{equation}
\label{relativeenergy}
{L}(t) \equiv \int_{\Omega} \frac{1}{2}|{\bf u}|^2 + \frac{(v(t) - v_{\infty})^2}{2v_{\infty}} + \frac{(w(t) - v_{\infty})^2}{2v_{\infty}} + |\nabla (\phi(t) - \Phi)|^2 \,dx.
\end{equation}

\begin{lemma}[Weighted Poincar\'e Inequality]
\label{wpi}
Let $\Omega$ be a connected, open subset of $\mathbb{R}^n$
and $0 < a \leq b < \infty.$ 
Then there exists $\clabel{mypoinc} = \cref{mypoinc}(a,b,\Omega)$ such that 
if  $\rho \in {\bf H}^1$ satisfies
\begin{equation*}
a \leq \rho(x) \leq b, \mbox{ a.e. } x \in \Omega
\end{equation*} 
and $f \in {\bf H}^1$ satisfies 
\begin{equation*}
\int_\Omega f \,dx = 0,
\end{equation*}
then
\begin{equation*}
\int_\Omega f^2 \,dx \leq \cref{mypoinc} \int_\Omega \left| \nabla (f \rho)  \right|^2\,dx.
\end{equation*}
If $\rho$ is merely positive and $\rho^{-1}$ is integrable, then there
is $\cref{mypoinc} = \cref{mypoinc}(\rho, {\Omega})$ for which 
the same conclusion holds. 
\end{lemma}
\begin{proof}
Suppose that no such constant exists.  
Then there is a sequence of functions $\{f_i\}_{i=1}^{\infty}$ and $\{\rho_i\}_{i=1}^{\infty}$ in ${\bf H}^1(\Omega)$
with
\begin{equation*}
\int_\Omega f_i \,dx = 0,\quad
\|f_i\|_{{\bf L}^2(\Omega)}^2 \geq i \int_{\Omega} |\nabla(f_i \rho_i)|^2 \,dx.
\end{equation*}
Let $h_i = f_i/\|f_i\|_{{\bf L}^2(\Omega)}$ so that 
\begin{equation}
\label{normalize}
1 = \|h_i\|_{{\bf L}^2(\Omega)}^2 \geq i \int_{\Omega} |\nabla(h_i \rho_i)|^2 \,dx.
\end{equation}
Write $g_i  = h_i\rho_i.$  Then 
\begin{equation*}
\int_{\Omega} g_i^2 \,dx = \int_{\Omega} h_i^2 \rho_i^2 \,dx \leq b^2 
\end{equation*}
shows that $g_i$ is bounded in ${\bf H}^1(\Omega)$ and thus converges weakly to an element 
$g \in {\bf H}^1(\Omega).$   Fatou's lemma,
\begin{equation*}
\int_{\Omega} |\nabla g|^2 \,dx \leq \liminf_{i \rightarrow \infty} \int_{\Omega} |\nabla g_i|^2 \,dx \leq \lim_{i \rightarrow \infty} \frac{1}{i} = 0
\end{equation*}
and the connectedness of $\Omega$ shows that $g(x) = G$ a.e. for some constant $G.$ 
Moreover, $ \rho_i^{-1}$ are uniformly bounded in ${\bf L}^2$ and so we may extract a subsequence
(reindexed by $i$) converging weakly to $\sigma$ in ${\bf L}^2$ with $b^{-1} \leq \sigma \leq a^{-1}$ a.e.
Clearly,
\begin{equation*}
\lim_{i \rightarrow \infty} \int_{\Omega} \rho_i^{-1} g_i \,dx =
\int_{\Omega} \sigma  g  \,dx.
\end{equation*}
 Then
\begin{equation*}
\begin{split}
G\int_{\Omega} \sigma \,dx &= \int_{\Omega} \sigma g\,dx 
= \lim_{i\rightarrow 0} \int_{\Omega}  \rho_i^{-1} g_i \,dx \\
&= \lim_{i\rightarrow 0} \int_{\Omega}  h_i \,dx= \lim_{i\rightarrow 0} \frac{1}{\|f_i\|_{{\bf L}^2(\Omega)}}\int_{\Omega}  f_i \,dx = 0.
\end{split}
\end{equation*}
We infer $G = 0$ since $\rho^{-1}$ is nonzero on a set of positive measure.  
The contradiction with \eqref{normalize} gives the existence of $\cref{mypoinc} = \cref{mypoinc}(a,b,\Omega).$

The existence of $\cref{mypoinc} = \cref{mypoinc}(\rho, \Omega)$ follows from the above proof by setting $\rho_i = \rho$ for each $i.$
\end{proof}

\begin{lemma}
\label{Ldiff}
Let $\Omega$ be a bounded subset of $\mathbf{R}^n$ with smooth 
boundary.  Then there exists positive constants $\clabel{Ldecay1}, \clabel{Ldecay2}$ and $\clabel{Ldecay3}$ 
depending only on ${\Omega},$ $v_{\infty}$ and $w_{\infty}$ such that
if $\langle \mathbf{u}, v,w,\phi \rangle$ is a weak solution of \eqref{NS}-\eqref{D}
on $Q_{T}$ then 
\begin{gather}
\begin{aligned}
\label{Ldecay}
\frac{dL}{dt}  \leq - \cref{Ldecay1} {L} + \cref{Ldecay2} {L}^2 + 
\cref{Ldecay3} \|\nabla(\phi - \phi_{\infty})\|_{\mathbf{L}^2}^2,\\ \quad
\mbox{ a.e. } t \in (0, T).
\end{aligned}
\end{gather}
\end{lemma}
\begin{remark}
\label{excuse}
The following proof applies equally well to the modified Galerkin approximation
from lemma \ref{smalltime}.
\end{remark}
\begin{proof}
For the sake of clarity, we will assume that $w \equiv w_{\infty} \equiv 0.$
The general case requires only minor modifications.
Define
\begin{equation*}
e = v - w_{\infty}, \quad \psi = \phi - \phi_{\infty}.
\end{equation*}
Using \eqref{poisson} and \eqref{PB}, we compute 
\begin{equation*}
\begin{aligned}
\frac{dL}{dt} &= \int_{\Omega} \mathbf{u}\cdot \mathbf{u}_t + \frac{e}{v_{\infty}}e_t + \nabla \psi \cdot \nabla \psi_t \,\mathrm{d}x\\
&= \int_{\Omega} -|\nabla \mathbf{u}|^2 + \mathbf{u}\cdot \nabla \phi \Delta \phi\,\mathrm{d}x  + 
\int_{\Omega}\left(\frac{e}{v_{\infty}}  - \psi \right)v_t \,\mathrm{d}x
= A + B.
\end{aligned}
\end{equation*}
Write 
\begin{equation*}
A_1 = \int_{\Omega} -|\nabla \mathbf{u}|^2\,\mathrm{d}x,
\quad
A_2 = \int_{\Omega} \mathbf{u}\cdot \nabla \phi \Delta \phi\,\mathrm{d}x.  
\end{equation*}
Note that by corollary \ref{statexist}, $\nabla v_{\infty} = v_{\infty} \nabla \phi_{\infty}.$ 
Note also that 
\begin{equation}
\label{niceidentity}
\nabla v 
=v_{\infty}\nabla \left(\frac{e}{v_{\infty}}\right)
+ v \nabla \phi_{\infty}.
\end{equation}
Using the fact that $v$ is a weak solution, 
\begin{equation*}
B = -\int_{\Omega} \nabla\left(\frac{e}{v_{\infty}}  - \psi \right)\cdot(\nabla v - v \nabla \phi - v\mathbf{u}) \,\mathrm{d}x.
\end{equation*}
Then, computing with \eqref{niceidentity},
\begin{equation*}
\begin{aligned}
B 
&= -\int_{\Omega} v_{\infty}\left|\nabla \left(\frac{e}{v_{\infty}}\right)\right|^2 \,\mathrm{d}x\\
& + \int_{\Omega} e\nabla\left(\frac{e}{v_{\infty}}\right)\cdot \nabla \psi 
+ 2v_{\infty} \nabla\left(\frac{e}{v_{\infty}}\right)\cdot \nabla \psi 
 \,\mathrm{d}x\\
& -\int_{\Omega} v |\nabla \psi|^2 \,\mathrm{d}x\\
&+\int_{\Omega} \left\{v\nabla\left(\frac{e}{v_{\infty}}\right) - v\nabla \psi\right\}\cdot \mathbf{u} \, \mathrm{d}x\\
& = B_1 + B_2 + B_3 + B_4.
\end{aligned}
\end{equation*}
We proceed by bounding $A_2, B_2,B_3$ and $B_4$ in terms
of $A_1$ and $B_1$ and integrals of higher powers of $|e|$ and $|\nabla \psi|.$
\begin{remark}
The presence of transport of the charges by the velocity
make the following calculation somewhat more subtle than
the analogous analysis for the Debye H\"uckel system, c.f. \cite{BeMeVa04}.
As will be demonstrated immediately below, the net exchange of kinetic energy $\frac{1}{2}|\mathbf{u}|^2$ 
and the relative energy $\frac{e^2}{2v_{\infty}}$ is a second order contribution.
This is in agreement with the cancelation of the entropy production due to 
transport and the kinetic energy production due to forcing seen in the derivation 
of the basic energy law
\eqref{apriori}
\end{remark}
We have 
\begin{equation*}
A_2 + B_4 
=
\int_{\Omega}\mathbf{u}
\cdot\left(\Delta \phi \nabla \phi 
+ \Delta \phi \nabla \left(\frac{e}{v_{\infty}}\right)
- \Delta \phi \nabla \psi\right) \, \mathrm{d}x.
\end{equation*}
Adding and subtracting $\Delta \phi_{\infty} \nabla \left(\frac{e}{v_{\infty}}\right)$
and using the relation $v = \Delta \phi, v_{\infty} = \Delta \phi_{\infty},$
\begin{equation*}
A_2 + B_4 =
\int_{\Omega} \mathbf{u}
\cdot\left(\Delta \phi \nabla \phi_{\infty} 
+ \Delta\phi_{\infty} \nabla \left(\frac{e}{v_{\infty}}\right)
+ e \nabla \left(\frac{e}{v_{\infty}}\right)
\right) \, \mathrm{d}x.
\end{equation*}
Finally, using the relation \eqref{niceidentity}, 
\begin{equation*}
A_2 + B_4
= \int_{\Omega} \mathbf{u}
\cdot\left( \nabla v 
+ e  \nabla \left(\frac{e}{v_{\infty}}\right)\right)
\, \mathrm{d}x.
\end{equation*}
Because $\mathbf{u}$ is divergence free and so is orthogonal to $\nabla v$ in $\mathbf{L}^2,$
the first product in the integrand vanishes. By Young's inequality, 
\begin{equation}
\label{firstABest}
A_2 + B_4 
\leq -a_1 B_1 + a_2\int_{\Omega} |\mathbf{u}|^2e^2 \,\mathrm{d} x
\end{equation}
where $a_1$ and $a_2$ are positive and $a_1a_2 = 4.$

The term $B_3$ is nonpositive.

To estimate $B_2$ we first note that $v_{\infty}$ is bounded. 
Thus there exists a $C_2$ depending only on $v_{\infty}$ 
for which 
\begin{equation}
\label{secABest}
B_2 \leq -(b_1 + 2d_1) B_1 + b_2\int_{\Omega}e^2|\nabla \psi|^2 \,\mathrm{d}x
+ 2C_2 d_2\int_{\Omega} |\nabla \psi|^2 \,\mathrm{d}x.
\end{equation} 
where $b_1,b_2,d_1,d_2$ are positive and 
$b_1b_2 = d_1d_2 = 4.$

Let 
$\rho = \frac{1}{v_{\infty}}$ and $f = e.$ 
By corollary \ref{statexist}, 
$\rho$ satisfies the first
hypothesis of lemma \ref{wpi} and the integral
of $e$ is zero, satisfying the second hypothesis of the lemma.
Hence there is a constant $\cref{mypoinc}$ depending only
on $v_{\infty}$ and $\Omega$ for which
\begin{equation*}
\int_{\Omega} e^2 \,\mathrm{d}x \leq 
\cref{mypoinc}\int_{\Omega} \left|\nabla \left(\frac{e}{v_{\infty}}\right)\right|^2 \,\mathrm{d}x.
\end{equation*}
Then, by corollary \ref{statexist}, there is $C_3 = C_3(v_{\infty}, \cref{mypoinc})$
for which 
\begin{equation}
\label{est3}
\int_{\Omega} \frac{2e^2}{v_{\infty}} \,\mathrm{d}x \leq 
C_3 \int_{\Omega} v_{\infty} \left|\nabla \left(\frac{e}{v_{\infty}}\right)\right|^2 \,\mathrm{d}x
= -C_3B_1.  
\end{equation}

Similarly, since $\psi \in \mathbf{H}^1_0$ is a solution of the Poisson equation with right hand side 
$e,$ there is also a $C_4 = C_4(v_{\infty}, \Omega, \cref{mypoinc})$  for which
\begin{equation}
\label{est4}
\int_{\Omega} \frac{1}{2}|\nabla \psi|^2 \,\mathrm{d}x 
\leq -C_4 B_1.
\end{equation}

Finally, by the Poincar\'e inequality, there is $C_5 = C_5(\Omega)$ for which 
\begin{equation}
\label{est5}
\int_{\Omega} \frac{1}{2}|\mathbf{u}|^2 \,\mathrm{d}x 
\leq -C_5 A_1.
\end{equation}
Adding \eqref{est3},\eqref{est4} and \eqref{est5} together, we have shown
that there is a positive constant $C_6 = C_6(v_{\infty}, \Omega, \cref{mypoinc})$ for which
\begin{equation}
\label{est6}
L \leq -C_6(A_1 + B_1).
\end{equation}

Arguing in a similar fashion, 
there is clearly a $C_7 = C_7(\Omega, v_{\infty},a_2,b_2)$ for which
\begin{equation}
\label{est7}
\int_{\Omega}
a_2 |u|^2 e^2 + b_2 e^2|\nabla \psi|^2 \,\mathrm{d}x \leq C_7 L^2.
\end{equation}
By \eqref{firstABest},\eqref{secABest} and \eqref{est7}, we have 
\begin{equation*}
\begin{aligned}
\frac{dL}{dt} &= A_1 + A_2 + B_1 + B_2 + B_3 + B_4 \\
&\leq A_1 + (1 - a_1 - b_1 -2d_1)B_1 + 2C_2d_2 \int_{\Omega} |\nabla \psi|^2 \,\mathrm{d}x
+ C_7L^2.
\end{aligned}
\end{equation*}
Choose $a_1, b_1$ and $d_1$ so that $a_1 - b_1 - 2d_2 = \frac{1}{2},$ thereby fixing $a_2,b_2,d_2, C_7.$  
Then, using \eqref{est6} and the fact that $A_1 + B_1 \leq 0,$ 
\begin{equation*}
\frac{dL}{dt} \leq 
-\frac{1}{2C_6}L  + C_7 L^2 + 2C_2d_2 \int_{\Omega} |\nabla \psi|^2 \,\mathrm{d}x.
\end{equation*}
Setting $\cref{Ldecay1} = \frac{1}{2C_6}, \cref{Ldecay2} = C_7$ and $\cref{Ldecay3} = 2C_2d_2,$ 
the lemma is now proved.
\end{proof}

From theorem \ref{PBThm}
there is $\rho_1 > 0$ so that for all $0< \rho_0 < \rho_1,$
\begin{equation*}
|\phi_{\infty}(x)| \leq 1, \quad \forall x \in \Omega.
\end{equation*}
Following the proof of lemma \eqref{Ldiff}
and using corollary \ref{statexist}, 
one may keep track of the constants 
$\cref{Ldecay1}, \cref{Ldecay2}$ and $\cref{Ldecay3}$ to prove
\begin{corollary}
\label{awesome}
There exist positive constants $\cref{Ldecay1}', \cref{Ldecay2}', \cref{Ldecay3}'$ 
and $\rho_1$ depending only on $\Omega$
such that if $\rho_0 < \rho_1$ then 
\begin{equation*}
\begin{aligned}
\frac{dL}{dt}
\leq -\cref{Ldecay1}' L 
+ \rho_0 \cref{Ldecay2}' L^2 + \rho_0 \cref{Ldecay3}'\|\nabla(\phi - \phi_{\infty})\|_{\mathbf{L}^2}^2,\\
\forall t \in [0,T_0].
\end{aligned}
\end{equation*}
\end{corollary}

\subsection{Proof of theorem \ref{l2ass}}
We are assuming $\Omega \subset \mathbf{R}^n, n = 2,3$
is bounded with smooth boundary.   
We will first use lemma \ref{Ldiff} to prove an extension 
property analogous to theorem \ref{mglobal}.

Let $\mathbf{u_0} \in \mathbf{H}, v_0, w_0 \in \mathbf{L}^2.$ 
Let $\{v^{h}_0, w^{h}_0\}$ 
be a sequence of functions satisfying \eqref{smoothness}
with $v^{h}_0, w^{h}_0 \rightarrow v_0, w_0$ in $\mathbf{L}^2$ as 
$h \downarrow 0.$  
Let $\langle \mathbf{u}_m^{h}, v_m^{h}, w_m^{h},\phi_m^h\rangle$ be the 
local modified-Galerkin approximate solution on $\mathbf{Q}_{T_0}$ 
obtained from 
lemma \ref{smalltime} with initial data $\mathbf{u}_0, v^{h}_0, w^{h}_0.$

Let $\rho_1$ be the constant from corollary \ref{awesome}.
Since $\|\nabla(\phi^h_m - \phi_{\infty})\|_{\mathbf{L}^2}^2 \leq 2L,$ we may 
choose 
\begin{equation*}
\rho_2 = \frac{\cref{Ldecay1}'}{8\cref{Ldecay3}'}
\end{equation*}
so that if $\rho_0 < \min\{\rho_1,\rho_2\}$ then 
\begin{equation*}
\begin{aligned}
\frac{dL}{dt}
\leq -\frac{3\cref{Ldecay1}'}{4} L 
+ \rho_0 \cref{Ldecay2}' L^2, \quad 
\mbox{ a.e. } t \in [0,T_0].
\end{aligned}
\end{equation*}
(see remark \ref{excuse}.)
Finally, let 
\begin{equation*}
\delta = \frac{\cref{Ldecay1}'}{4\rho_0 \cref{Ldecay2}'}.
\end{equation*}
Then 
\begin{equation}
\label{Ldecayexp}
L(0) < \delta \mbox{ implies } L(t) < \delta e^{-t\frac{\cref{Ldecay1}'}{2}}, \quad \mbox{ a.e. }t \in [0,T_0].
\end{equation}

This inequality implies
that the $\mathbf{L}^2$ norm of $\mathbf{u}_m^h, v_m^h$ and $w_m^h$ remains bounded independently of $t.$
The extension property now follows exactly as in theorem \ref{mglobal}.
One checks (just as in the proofs of estimates 
\eqref{uvwbound} and 
\eqref{fpbound}) that there are constants 
$\clabel{clem3d}, \clabel{clem3d2}$ 
independent of $m$ and $h$ for which 
\begin{equation*}
\begin{aligned}
\sup_{t \in (0,T)} \|\mathbf{u}_m^h, v_m^h, w_m^h\|_{\mathbf{L}^2}
+ \int_0^T \|\nabla \mathbf{u}_m^h, \nabla v_m^h, \nabla w_m^h\|_{\mathbf{L}^2}^2 \,ds \\
+ \int_0^T \left\|\frac{d\mathbf{u}_m^h}{dt}\right\|_{\mathbf{V}^*}^{\frac{4}{3}} 
+  \left\|\frac{dv_m^h}{dt}, \frac{dw_m^h}{dt}\right\|_{\mathbf{H}^{-1}}^{\frac{4}{3}}\,ds 
\leq \cref{clem3d}.
\end{aligned}
\end{equation*}
and 
\begin{equation*}
\sup_{t \in (0,T)} \| \mathbf{f}_m^h \|_{\mathbf{V}^*}^2
+ \int_0^T \left \|\frac{d\mathbf{f}_m^h}{d t}\right\|_{\mathbf{V}^*}^{\frac{4}{3}} \,ds \leq \cref{clem3d2}
\end{equation*}
where 
\begin{equation*}
\mathbf{f}_m = \Delta \phi_m^h \nabla \phi_m^h.
\end{equation*}
Letting $h \rightarrow 0, m \rightarrow \infty$ we see 
that some subsequence of 
$\langle \mathbf{u}^h_m, v^h_m,w^h_m,\phi^h_m \rangle$
converges to a global weak solution $\langle \mathbf{u}, v,w,\phi \rangle$ of \eqref{NS}-\eqref{I}.
This proves the first part of theorem \ref{l2ass}.

Note that $\mathscr{E}_2 = 2 L.$ 
From \eqref{Ldecayexp}, let 
$\epsilon_2 = 2\delta.$  
If $\mathscr{E}_2(0) < \epsilon_2,$ then $L(0) < \delta$ 
and we may apply \eqref{Ldecayexp} for a.e. $t \in [0,T].$ 
Let $\lambda_2 = \frac{\cref{Ldecay1}'}{2}$ which depends only on $\Omega.$  
Then the global weak solution $\langle \mathbf{u}, v,w,\phi \rangle$
satisfies 
\begin{equation*}
\mathscr{E}_2 = 2 L \leq 2 \delta e^{-t \lambda_2}, \quad \mbox{ a.e. } t \in [0,T].
\end{equation*}

The last part of the theorem now follows immediately from the following lemma.
Using the embedding $\mathbf{V} \subset \mathbf{H},$
minor modifications of the proof of \cite{BiHeNa94}, theorem 6
gives
\begin{lemma}
Let $\Omega \subset \mathbb{R}^n, n = 2,3$ be a bounded, open
set with smooth boundary.  
If $\langle \mathbf{u}, v, w, \phi\rangle$
satisfy \eqref{poisson}, \eqref{D}, 
\begin{gather*}
\sup_{t \in [0, \infty)}
\|v,w\|_{\mathbf{L}^2} < \infty,\\
\sup_{t \in [0, \infty)}
\int_{\mathbf{Q}_t}
v|\nabla\log(ve^{-\phi})|^2+ 
w|\nabla\log(we^{\phi})|^2 + |\nabla \mathbf{u}|^2\,\mathrm{d}x < \infty,
\end{gather*}
then there is some sequence $t_j \rightarrow \infty$
for which
\begin{equation*}
\lim_{j \rightarrow \infty} L(t_j) = 0.
\end{equation*}
\end{lemma}

\subsection{Proof of theorem \ref{globalex3d}}
Let $\mathbf{u_0} \in \mathbf{V} \cap \mathbf{H}^2, v_0, w_0 \in \mathbf{H}^2$
and $T > 0.$  
Let $\{v^{h}_0, w^{h}_0\}$ 
be a sequence of functions satisfying \eqref{smoothness}
with $v^{h}_0, w^{h}_0 \rightarrow v_0, w_0$ in $\mathbf{H}^2$ as 
$h \downarrow 0.$  
Let $\langle \mathbf{u}_m^{h}, v_m^{h}, w_m^{h},\phi_m^h\rangle$ be the 
local modified-Galerkin approximate solution on $\mathbf{Q}_{T_0}$ 
obtained from 
lemma \ref{smalltime} with initial data $\mathbf{u}_0, v^{h}_0, w^{h}_0.$
Let $\rho_2$ and $\epsilon_2$ be the constants from theorem \ref{l2ass}.
Assume that $\rho_0 < \rho_3$ and $\mathscr{E}_2 < \epsilon_3$
where  $\rho_3 < \rho_2$ and $\epsilon_3 < \epsilon_2$ will be determined below.
Applying the results from theorem \ref{l2ass}, $\langle \mathbf{u}_m^{h}, v_m^{h}, w_m^{h},\phi_m^h\rangle$
is defined for all $t \in [0,T]$ and some subsequence converges to a global weak solution of \eqref{NS}-\eqref{I}
as $m \rightarrow \infty $ and $h \rightarrow 0.$
In the sequel we suppress the sub- and superscripts $m$ and $h.$ 

From lemma \ref{smalltime} 
we infer that $v_t$ and $w_t$ are smooth in $x.$ 
In particular, $v_t$ and $w_t$ are classical
solutions of 
\begin{gather*}
v_{tt} + \mathbf{u}_t \cdot \nabla v
+ \mathbf{u} \cdot \nabla v_t = \nabla \cdot 
(\nabla v_t - v_t \nabla \phi  - v \nabla \phi_t),\\
w_{tt} + \mathbf{u}_t \cdot \nabla w
+ \mathbf{u} \cdot \nabla w_t = \nabla \cdot 
(\nabla w_t + w_t \nabla \phi  + w \nabla \phi_t).
\end{gather*}
Multipling these equations by $v_t$ and $w_t$ and integrate over $\Omega.$
Integrating by parts and noting that the boundary terms vanish 
in this case as well,
\begin{equation*}
\frac{1}{2}\frac{d}{dt}
\left\|\frac{\partial v}{\partial t}\right\|_{\mathbf{L}^2}^2
+ \left\|\nabla \frac{\partial v}{\partial t}\right\|_{\mathbf{L}^2}^2 
\leq 
\left(\Phi, \nabla \frac{\partial v}{\partial t}\right) 
\end{equation*}
where 
\begin{equation*}
\Phi = \frac{\partial \mathbf{u}}{\partial t} v + \nabla \frac{\partial \phi}{\partial t} v 
+ \nabla \phi \frac{\partial v}{\partial t}.
\end{equation*}
We estimate the $\mathbf{L}^2$ norm of $\Phi.$ 
From \eqref{conserved}, the integral of $v_t$ vanishes. 
By the Sobolev embedding 
$\mathbf{H}^1 \subset \mathbf{L}^4$
and the Poincare inequality, the estimate
\begin{equation*}
\|\Phi\|_{\mathbf{L}^2}^2 
\leq 
C\left(\left\|\nabla \frac{\partial v}{\partial t},\nabla \frac{\partial v}{\partial t}
\right\|_{\mathbf{L}^2}^2 
+\left\|\nabla \frac{\partial \mathbf{u}}{\partial t}\right\|_{\mathbf{L}^2}^2 \right)
\|v, w\|_{\mathbf{H}^1}^2 
\end{equation*}
for some $C = C(\Omega)$ is straightforward. The analogous estimate holds for 
$w_t.$   

Now differentiate \eqref{ode} with respect to $t$ and
multiply the resulting system componentwise by $\dot{u}_i$
and add the equations for $i = 1,\dots, m.$   One concludes
\begin{equation*}
\frac{1}{2}\frac{d}{dt}
\left\|\frac{\partial \mathbf{u}}{\partial t}\right\|_{\mathbf{L}^2}^2
+ \left\|\nabla \frac{\partial \mathbf{u}}{\partial t}\right\|_{\mathbf{L}^2}^2 
= 
b\left(\frac{\partial \mathbf{u}}{\partial t},
\mathbf{u},\frac{\partial \mathbf{u}}{\partial t}\right) 
+ \left(\Psi, \nabla \frac{\partial v}{\partial t}\right) 
\end{equation*}
where
\begin{equation*}
\Psi = \nabla \phi \otimes \nabla \frac{\partial \phi}{\partial t}
+ \nabla \frac{\partial \phi}{\partial t} \otimes \nabla \phi. 
\end{equation*}
The usual estimate in the small data regularity proof for Navier-Stokes
shows that 
\begin{equation*}
b\left(\frac{\partial \mathbf{u}}{\partial t},
\mathbf{u},\frac{\partial \mathbf{u}}{\partial t}\right) 
\leq 
\left\|\nabla \frac{\partial \mathbf{u}}{\partial t}\right\|_{\mathbf{L}^2}^2 
\|\nabla \mathbf{u}\|_{\mathbf{L}^2}. 
\end{equation*}
Using the embedding $\mathbf{H}^{-1} \hookrightarrow \mathbf{H}^1$ 
from the 
Poisson equation, there is $C = C(\Omega)$ for which  
\begin{equation*}
\|\Psi\|_{\mathbf{L}^2}^2
\leq 
C\left\|\nabla \frac{\partial v}{\partial t},\nabla \frac{\partial v}{\partial t}
\right\|_{\mathbf{L}^2}^2 
\|v, w\|_{\mathbf{H}^1}^2 
\end{equation*}

Define 
\begin{gather*}
G(t) = 
\left\|\frac{\partial \mathbf{u}}{\partial t},
\frac{\partial v}{\partial t},
\frac{\partial w}{\partial t}\right\|_{\mathbf{L}^2}^2,\\
H(t) =
\left\|\frac{\partial \mathbf{u}}{\partial t},
\frac{\partial v}{\partial t},
\frac{\partial w}{\partial t}\right\|_{\mathbf{H}^1}^2,\\
I(t) = \|\mathbf{u}, v, w\|_{\mathbf{H}^1}^2. 
\end{gather*}
The above estimates and the Poincar\'e inequality show that
\begin{equation*}
\frac{d}{dt} G(t)  + H(t)( 1 - I(t) ) \leq 0, \quad \forall t \in [0,T].
\end{equation*}

Now we develope a relationship between $I(t)$ and $G(t).$
From \eqref{ode}, 
\begin{equation*}
\begin{aligned}
2\|\nabla \mathbf{u}\|_{\mathbf{L}^2}^2
&= -2(\nabla\phi \otimes \nabla \phi, \nabla \mathbf{u}) - 
2\left(\frac{\partial \mathbf{u}}{\partial t}, \mathbf{u}\right)\\
&\leq \|\nabla \mathbf{u}\|_{\mathbf{L}^2}^2
+  
\|\nabla \phi\|_{\mathbf{L}^4}^2 + 
\left\|\frac{\partial \mathbf{u}}{\partial t}\right\|_{\mathbf{L}^2}\|\mathbf{u}\|_{\mathbf{L}^2}.
\end{aligned}
\end{equation*}
Define 
\begin{equation*}
\clabel{final1} = \sup_{t \in (0,T)} \|\nabla \phi\|_{\mathbf{L}^4}^2,\quad 
\clabel{final2} = \sup_{t \in (0,T)} \|\mathbf{u}\|_{\mathbf{L}^2}.
\end{equation*}
Then
\begin{equation*}
\|\nabla \mathbf{u}\|_{\mathbf{H}^1}
\leq \cref{final1} + \cref{final2} G(t)^{\frac{1}{2}}, \quad \forall t \in [0,T].
\end{equation*}
Similarly, by \eqref{conserved} and the triangular inequality,  
\begin{equation*}
\begin{aligned}
\|v\|_{\mathbf{H}^1}^2
&\leq  
2\left(\int_{\Omega} v_0 \,\mathrm{d}x\right)^2
+ 3\|\nabla v\|_{\mathbf{L}^2}^2 \\
&= 
2\rho_0^2 
- 3\left(\frac{\partial v}{\partial t}, v\right)
+ 3(v \nabla \phi ,\nabla v)\\
&\leq 
2 \rho_0^2  
+ 3\left\|\frac{\partial v}{\partial t}\right\|_{\mathbf{L}^2}
\|v\|_{\mathbf{L}^2}
+ C_1 \cref{final1} \|v\|_{\mathbf{H}^1}^2
\end{aligned}
\end{equation*}
for some $C_1 = C_1(\Omega).$ 
Applying a similar estimate the $w,$ we find
\begin{equation*}
\|v, w\|_{\mathbf{H}^1}^2
\leq \frac{1}{1 - C_1 \cref{final1}}
\left(
4\rho_0^2 
+ \clabel{final3} \left\|\frac{\partial v}{\partial t},
\frac{\partial w}{\partial t}\right\|_{\mathbf{L}^2}
\right)
\end{equation*}
where 
\begin{equation*}
\cref{final3} = 3\sup_{t \in (0,T)} \|v, w\|_{\mathbf{L}^2}.
\end{equation*}
Define
\begin{equation*}
\clabel{final4} =
\frac{4\rho_0^2}{1 - C_1 \cref{final1}} 
+ \cref{final1},
\quad  
\clabel{final5} =
\frac{\cref{final3} }{1 - C_1 \cref{final1}} 
+ \cref{final2}.
\end{equation*}
We have shown that 
\begin{equation*}
I(t) \leq \cref{final4} + \cref{final5} G(t)^{\frac{1}{2}}.
\end{equation*}

Thus,
\begin{equation*}
\frac{d}{dt} G(t)  + H(t)
\left( 1 - \cref{final4} -  
\cref{final5} G(t)^{\frac{1}{2}} \right) 
\leq 0, \quad \forall t \in [0,T].
\end{equation*}
Assume for the moment that $\rho_3, \epsilon_3$ and $\delta_3$ 
can be chosen so that 
\begin{equation}
\label{finalneed}
1 - \cref{final4} -  
\cref{final5} G(0)^{\frac{1}{2}} > 0.
\end{equation}
independently of $h$ and $m.$ 
If follows that 
\begin{equation}
\label{regbound}
G(t) + \int_0^t H(s) \,ds \leq G(0), \quad t \in [0,T]
\end{equation} 
for all $m$ and $h.$ 
Letting $h \rightarrow 0$ and $m \rightarrow \infty$
we may extract a subsequence of
$\langle \mathbf{u}_m^h,v_m^h, w_m^h, \phi_m^h \rangle$ 
which converges to a 
global weak solution 
$\langle \mathbf{u},v, w, \phi \rangle$ 
of \eqref{NS}-\eqref{I}.  
From \eqref{regbound}, this solution satisfies the estimate
\begin{equation*}
\mathbf{u}\in \mathbf{L}^{\infty}((0,T); \mathbf{H}^2),
\quad
\mathbf{u}_t\in \mathbf{L}^{\infty}((0,T); \mathbf{L}^2).
\end{equation*}
Arguing as in the end of the proof of theorem \eqref{globalex},
we find that $\mathbf{u}$ is H\"older continuous on $\mathbf{Q}^T$ 
and $\mathbf{C}^{2+\alpha}$ 
on compact subsets of $\mathbf{Q}_T$
and $v,w$ are $\mathbf{C}^{2+\alpha}$ on $\mathbf{Q}_{s,t}$ 
for any $0 < s < t \leq T.$   The standard arguments show that 
$\langle \mathbf{u}, v,w,\phi \rangle$ is unique, c.f. \cite{TEMAM01}.

Now we show that $\rho_3, \epsilon_3$ and $\delta_3$  may be chosen in order that
\eqref{finalneed} be satisfied.  
Note that \eqref{finalneed} holds provided $\rho_0, \cref{final1}, \cref{final2}$
and $\cref{final3}$ are sufficiently small and $G(0)$ is sufficiently small
with respect to $\frac{\cref{final4}}{\cref{final5}}.$ 

By the regularity of solutions to the Poisson equation,
there is $C_2 = C_2(\Omega)$ so that 
\begin{equation*}
\cref{final1} \leq C_2 \cref{final3}^2.
\end{equation*}
By the triangular inequality, 
\begin{equation*}
\begin{aligned}
\cref{final3}
&\leq \sup_{t \in [0,T]}3 \{\|v - v_{\infty},w - w_{\infty}\|_{\mathbf{L}^2}
+  \|v_{\infty},w_{\infty} \|_{\mathbf{L}^2}\}\\
&\leq \sup_{t \in [0,T]} C_3\rho_0(\sqrt{L} + 1)
\end{aligned}
\end{equation*}
for some $C_3 = C_3(\Omega),$
provided $\rho_0 < \rho_1.$  
Finally, 
\begin{equation*}
\cref{final2} \leq  2\sup_{t \in [0,T]}\sqrt{L}.
\end{equation*}

By assumption, $\rho_0 < \rho_3 < \rho_2$ and $\mathscr{E}_2 < \epsilon_3 < \epsilon_2.$
By theorem \ref{l2ass}, we have then that 
\begin{equation*}
\sup_{t \in [0,T]} L = \sup_{t \in [0,T]} \frac{1}{2}\mathscr{E}_2 < \epsilon_3.
\end{equation*}
Using the above bounds on $\cref{final1}, \cref{final2}$ and $\cref{final3}$
in terms of $L$ and the bound on $L$ in terms of $\epsilon_3,$ 
we may choose $\epsilon_3$ and $\rho_3$ so that 
\begin{equation*}
\begin{aligned}
\cref{final4} &\leq
\frac{4\rho_0^2}{1 - C_1C_2 C_3^2\rho_0^2 (\epsilon_3 +1 )} +  
C_2C_3^2 \rho_0^2 (\epsilon_3 + 1) < \frac{1}{2},\\
\cref{final5} &\leq
\frac{C_3\rho_0(\sqrt{\epsilon_3} + 1)}{1 - 6C_1C_2C_3\rho_0(\sqrt{\epsilon_3} + 1)} +  
2\sqrt{\epsilon_3} < 1,\\
G(0) &\leq \frac{1}{4}.
\end{aligned}
\end{equation*}
This implies \eqref{finalneed}.
Certainly there exists $C_4 = C_4(\Omega)$ so that 
\begin{equation*}
G(0) \leq C_4 \|\mathbf{u}_0, v_0, w_0\|_{\mathbf{H}^2}^4. 
\end{equation*}
Setting $\delta_3 = \frac{1}{\sqrt{4C_4}}$ now completes the proof.

\section{Conclusion}
The equations of a viscous, incompressible fluid
coupled with diffuse charges in two and three dimensions have been studied.  
The key step 
toward the existence of global in time solutions is the presence 
of a decaying entropy function which guarantees the dissipation
of kinetic and electrostatic energy and entropy.  

The most serious obstruction to formulating a global
existence result like theorem \ref{globalex} when 
$\mathrm{dim}\Omega = 3$ is the Debye-H\"uckel 
system.  A different approach is to consider very \emph{weak}
solutions, i.e. those where $v,w \in \mathbf{L}^{\infty}((0,T);\mathbf{L}^1)$
satisfy the weak formulation in terms of test functions $\omega \in \mathbf{C}^1(\Omega).$  
How one defines the forcing term $\Delta \phi \nabla \phi$ and a solution
of \eqref{poisson} then becomes a more delicate matter. 

The techniques used in this paper certainly apply to other Dirichlet 
conditions than \eqref{D} and other 2nd order elliptic operators than the Laplacian. 
A future avenue of study are electrorheological fluids where 
the charge is vector valued and the potential is the polarization
potential, see \cite{ZhGoLiWeSh08}.

\newcommand{\etalchar}[1]{$^{#1}$}

\end{document}